\documentclass[reqno, 11pt]{amsart}

\usepackage[english]{babel}
\usepackage{amsmath}
\usepackage{amssymb}
\usepackage{enumerate}
\usepackage{ifthen}
\usepackage{bbm}
\usepackage{color}
\usepackage{geometry}
\provideboolean{shownotes} 
\setboolean{shownotes}{true}
\usepackage{color}
\usepackage{hyperref}
\newtheorem{theorem}{Theorem}[section]

\newtheorem{lemma}[theorem]{Lemma}

\newtheorem{definition}[theorem]{Definition}

\theoremstyle{remark}
\newtheorem{remark}[theorem]{Remark}

\newcommand{\e}{\varepsilon}		       
\newcommand{\R}{\mathbb{R}}
\newcommand{\T}{\mathbb{T}^3}
\newcommand{\N}{\mathbb{N}}

\newcommand{\dive}{\mathop{\mathrm {div}}}

\newcommand{\rrho}{\sqrt{\rho}}
\newcommand{\la}{\lambda}
\newcommand{\rn}{\rho_{\e}}
\newcommand{\rrn}{\sqrt{\rho_{\e}}}
\newcommand{\un}{u_{\e}}
\newcommand{\wn}{w_{{\e}}}
\newcommand{\bedn}{\beta_{\delta}^{l}(u_{\e})}
\newcommand{\bedm}{\beta_{\delta}^{l}(v_m)}

\newcommand{\bed}{\beta_{\delta}^{l}(u)}
\newcommand{\beln}{\bar{\beta}_{\lambda}(\rho_{\e})}
\newcommand{\bel}{\bar{\beta}_{\lambda}(\rho)}
\newcommand{\belr}{\bar{\beta}_{\lambda}(\overline{\rho}_{r})}
\newcommand{\belrho}{\bar{\beta}_{\lambda}(\rho)}
\newcommand{\belrp}{\bar{\beta}^{'}_{\lambda}(\overline{\rho}_{r})}
\newcommand{\belrhop}{\bar{\beta}^{'}_{\lambda}(\rho)}
\newcommand{\belnp}{\bar{\beta}^{'}_{\lambda}(\rho_{\e})}
\newcommand{\regrho}{\overline{\rho}_{r}}
\newcommand{\nablay}{\nabla_{y}}

\newcommand{\re}{\rho_{\e}}
\newcommand{\ue}{u_{\e}}
\newcommand{\rre}{\sqrt{\rho_\e}}
\newcommand{\bhmr}{\hat{\beta}_{\delta}(\overline{v_m}_r)}
\newcommand{\bhm}{\hat{\beta}_{\delta}(v_m)}
\newcommand{\bhu}{\hat{\beta}_{\delta}(u)}
\newcommand{\bhue}{\hat{\beta}_{\delta}(u_{\e})}
\newcommand{\weakto}{\rightharpoonup}
\newcommand{\weaktos}{\stackrel{*}{\rightharpoonup}}
\newcommand{\MT}{\mathcal{T}}

\numberwithin{equation}{section}



\begin{document}

\title[Navier-Stokes-Korteweg Equations]{Global existence of weak solutions to the Navier-Stokes-Korteweg equations}

\author[P. Antonelli]{Paolo Antonelli}
\address[P. Antonelli]{\newline
GSSI - Gran Sasso Science Institute \\ Viale Francesco Crispi 7\\67100, L'Aquila \\Italy}
\email[]{\href{paolo.antonelli@gssi.it}{paolo.antonelli@gssi.it}}

\author[S. Spirito]{Stefano Spirito}
\address[S. Spirito]{\newline
DISIM- Dipartimento di Ingegneria e Science dell'Informazione e Matematica \\ Via Vetoio\\67100, L'Aquila \\Italy}
\email[]{\href{stefano.spirito@univaq.it}{stefano.spirito@univaq.it}}

\begin{abstract}
In this paper we consider the Navier-Stokes-Korteweg equations for a viscous compressible fluid with capillarity effects in three space dimensions. We prove global existence of finite energy weak solutions for large initial data. Contrary to previous results regarding this system, vacuum regions are allowed in the definition of weak solutions and no additional damping terms are considered. The convergence of the approximating solutions is obtained by introducing suitable truncations in  the momentum equations of the velocity field and the mass density at different scales and use only the {\em a priori} bounds obtained by the energy and the BD entropy. Moreover, the approximating solutions enjoy only a limited amount of regularity, and the derivation of the truncations of the velocity and the density is performed by a suitable regularization procedure.
\end{abstract}

\maketitle

\section{Introduction}\label{sec:intro}
The aim of this paper is to prove global existence of finite energy weak solutions of the following Navier-Stokes-Korteweg system in $(0, T)\times\T$:
\begin{align}
&\partial_t\rho+\dive(\rho u)=0,\,\rho\geq 0,\label{eq:qns1}\\
&\partial_t(\rho u)+\dive(\rho u\otimes u)+\nabla\rho^{\gamma}-\dive(\rho Du)-\rho\nabla\Delta\rho=0,\label{eq:qns2}
\end{align}
with initial data 
\begin{equation}\label{eq:id}
\begin{aligned}
\rho(0,x)&=\rho^0(x),\\
(\rho u)(0,x)&=\rho^0(x)u^0(x).
\end{aligned}
\end{equation} 
Here, $\T$ denotes the three-dimensional flat torus, the function $\rho$ represents the density of the fluid and the vector $u$ is the velocity field.\par
The system \eqref{eq:qns1}-\eqref{eq:qns2} is a particular case of a more general class of equations which in their general form read
\begin{equation}\label{eq:nskgeneral}
\begin{aligned}
&\partial_t\rho+\dive(\rho u)=0\\
&\partial_t(\rho u)+\dive(\rho u\otimes u)+\nabla p=\dive\mathbb S+\dive\mathbb K,
\end{aligned}
\end{equation}
where $\mathbb S$ is the viscosity stress tensor given by
\begin{equation}\label{eq:visc}
\mathbb S=h(\rho)\,D u+g(\rho)\dive u\mathbb I,
\end{equation}
the coefficients $h$ and $g$ satisfying
\begin{equation*}
\begin{aligned}
&h\ge 0,\quad&h+3g\geq0,
\end{aligned}
\end{equation*}
 and the capillarity term $\mathbb K$ is determined by
\begin{equation}\label{eq:cap}
\dive\mathbb K=\nabla\left(\rho\dive(k(\rho)\nabla\rho)-\frac12(\rho k'(\rho)-k(\rho))|\nabla\rho|^2\right)-\dive(k(\rho)\nabla\rho\otimes\nabla\rho).
\end{equation}
Notice that system \eqref{eq:qns1}-\eqref{eq:qns2} is obtained from \eqref{eq:nskgeneral}-\eqref{eq:cap} by choosing  $k(\rho)=1$, $h(\rho)=\rho$ and $g(\rho)=0$. The tensor $\mathbb{K}$ is also called the Korteweg tensor and is derived rigorously from thermodynamic considerations by Dunn and Serrin in \cite{DS}. Systems of Korteweg type arise in modeling several physical phenomena, {\em e.g.} capillarity phenomena in fluids with diffuse interfaces, where the density experiences steep but still smooth changes of values. \par
Local existence of smooth solutions and global existence with small data for the system \eqref{eq:qns1}-\eqref{eq:qns2} have been proved in \cite{HL, HL1} by using a fixed point argument. 
Regarding the study of global weak solutions for the general Navier-Stokes-Korteweg systems only few results are available in the literature as in the analysis several mathematical difficulties appear. Besides the strong nonlinearity in the higher order derivatives determined by the Korteweg term, a crucial difficulty is that the viscosity coefficients $h$ and $g$ may vanish in the vacuum region and hence the velocity field and its gradient do not fulfil satisfactory a priori bounds. This is opposite to what happens for instance in the case with constant viscosities, when the velocity field satisfies suitable Sobolev bounds.
On the other hand, in the case when the coefficients $h, g$ satisfy
\begin{equation*}
\begin{aligned}
&g(\rho)=\rho h'(\rho)-h(\rho),\quad  
\end{aligned}
\end{equation*}
and the capillarity coefficient $\kappa(\rho)$ is chosen in a suitable way, it is possible to derive an additional a priori estimate, introduced by Bresch-Desjardins in \cite{BD} and called BD entropy, which yield further regularity properties on the density. By exploiting such a priori estimates, in \cite{BDL} the authors show a global existence result for weak solutions to \eqref{eq:qns1}-\eqref{eq:qns2}, where the test function are taken as $\rho\phi$, with $\phi$ smooth and compactly supported. This particular notion of weak solutions has the advantage to avoid some mathematical difficulties which arise in the definition of the velocity field in the vacuum region.
The result was later extended in \cite{J} to the case of Quantum-Navier-Stokes, namely when the capillarity coefficients is given by $k(\rho)=1/\rho$ in \eqref{eq:cap}. 
When the system \eqref{eq:qns1}-\eqref{eq:qns2} is augmented by a damping term in the equation for the momentum density, it is possible to prove the existence of global solutions by using the standard notion of weak solutions \cite{BD}. Indeed the presence of the damping term allows to define the velocity field almost everywhere in the domain.\par
When dealing with general finite energy weak solutions to \eqref{eq:qns1}-\eqref{eq:qns2}, a major mathematical difficulty arises in defining the velocity field in the vacuum region, due to the degeneracy of the viscosity coefficient $h(\rho)=\rho$. The momentum density is always well defined, but unfortunately the standard a priori estimates given by the physical energy (and by the BD entropy) do not avoid a possible concentration which would prevent the convergence of the convective term in the compactness argument. 
This is also the case when dealing with the barotropic Navier-Stokes system with degenerate viscosity, namely system \eqref{eq:qns1}-\eqref{eq:qns2} without the capillarity term. In \cite{MV} a further integrability estimate was inferred for the velocity field, yielding the compactness of weak solutions. However in the presence of a capillarity term it seems that a Mellet-Vasseur \cite{MV} type estimate fails. In the special case when the viscosity and capillarity coefficient satisfies the relation
 \begin{equation}\label{eq:relation}
 k(\rho)=\frac{h'(\rho)^2}{\rho},
 \end{equation}
 then it is possible to overcome this difficulty by considering an auxiliary system written in terms of an effective velocity field where the Korteweg tensor vanishes. 
Note that this relation plays a crucial role in the theory, see for example \cite{BGL} where the authors study the vanishing viscosity limit for the Quantum Navier-Stokes equations, or \cite{AS1} where \eqref{eq:relation} is extensively exploited to construct the approximating system and \cite{BCNV} where numerical methods are performed.

We stress that in  \eqref{eq:qns1}-\eqref{eq:qns2} the viscosity and capillarity coefficients do not satisfy the relation \eqref{eq:relation} and hence in this paper we cannot rely on a similar analysis.

 In order to overcome this difficulty and to prove our global existence result, we exploit a truncation argument, in the spirit of DiPerna-Lions \cite{DPL} for linear continuity equations, as it is also exploited in \cite{LV}.
 In our approach, we start by truncating the velocity field, but an additional truncation of the density is needed because of the lack of control on the third order term.  The two different truncations are performed at different scales, which at the end will be optimized to prove the convergence of both the third-order term and the convective term. The formal argument is also explained in \cite{AS3}, where it is showed the compactness of solutions of \eqref{eq:qns1}-\eqref{eq:qns2}. In this paper, we show the global existence of weak solutions, see Subsection \ref{subsec:main} and Theorem \ref{teo:main} for the statement of the main result and the necessary definitions. Here, the main difficulty is to rigorously justify at the level of the approximation the compactness argument. In this regard, we consider the following approximation 
\begin{equation}\label{eq:appintro}
\begin{aligned}
&\partial_t\re+\dive(\re\ue)=0,\\
&\partial_{t}(\re\ue)+\dive(\re\ue\otimes\ue)-\dive(\re\,D\ue)+\nabla\re^{\gamma}+\e\re|\ue|^2\ue+\e\ue\\&=\re\nabla\Delta\re+\e\re\nabla\left(\frac{\Delta\sqrt{\re}}{\sqrt{\re}}\right).
\end{aligned}
\end{equation}

Notice that following the argument in \cite{VY2} and \cite{BD} it is possible to prove global existence of weak solutions. Unfortunately, due to the limited amount of regularity, see Definition \ref{def:aws}, is not possible to justify the truncations of the velocity field $\ue$ and the density $\re$. In this regard, we perform suitable regularization of the weak solutions of \eqref{eq:appintro} which allows us to justify the formal argument in \cite{AS3}, see the proof of Theorem \ref{teo:exmain} for more details.  Roughly speaking, to gain regularity in the velocity we truncate it where is close to vacuum and we derive an equation for the regularized velocity. This produces several errors in the equation. In order to control the one involving the third-order we need a further truncation of the density at the infinity. We conclude by pointing out that it would be interesting to provide an approximating system as in \cite{AS1, LX}, as it would provide smooth approximating solutions for the system \eqref{eq:qns1}-\eqref{eq:qns2}. \par
Finally, we give a brief account of the state of art of the analysis of the Cauchy problem  for the general system \eqref{eq:nskgeneral}-\eqref{eq:cap}.  In the case  $\kappa=0$ \eqref{eq:nskgeneral} reduces to the system of compressible Navier-Stokes equations. When the viscosity coefficient $h(\rho)$ is chosen degenerating on the vacuum region $\{\rho=0\}$ the Lions-Feireisl theory, \cite{L}, \cite{F}, and the recent approach in \cite{BJ} cannot be used since they rely on the Sobolev bound of the velocity field.  The global existence of weak solutions has been proved independently in \cite{VY1} and \cite{LX} in the case $h(\rho)=\rho$ and $g(\rho)=0$. In both cases, non trivial approximation procedures are required to prove the BD entropy and the Mellet-Vasseur inequality. When the viscosity $\nu=0$, the system \eqref{eq:nskgeneral} is called Euler-Korteweg. In \cite{BGDD} local well-posedness has been proved, while in \cite{AH} the global existence of smooth solutions with small data has been proved. Moreover, when $k(\rho)=1/\rho$ the system \eqref{eq:nskgeneral} is called Quantum Hydrodynamic system (QHD) and arises for example in the description of quantum fluids. The global existence of finite energy weak solutions for the QHD system has been proved in \cite{AM, AM2} without restrictions on the regularity or the size of the initial data. Non uniqueness results by using convex integration methods has been proved in \cite{DFM}.
Relative entropy methods to study singular limits for the equations \eqref{eq:nskgeneral}-\eqref{eq:cap} have been exploited in \cite{BGL, DFM, GLT, DM1}, in particular we mention the incompressible limit in \cite{AHM} in the quantum case, the quasineutral limit \cite{DM} for the constant capillarity case and the vanishing viscosity limit in \cite{BGL}. 
  Finally, the analysis of the long time behaviour for the isothermal Quantum-Navier-Stokes equations has been performed in \cite{CCH}.

 \subsection*{Organization of the paper.}
The paper is organized as follows. In Section \ref{sec:pre} we fix the notations, give the precise definition of weak solutions of \eqref{eq:qns1}-\eqref{eq:qns2} and we recall some of the main tools used in the proofs. In Section \ref{sec:app} we give the definition of weak solutions of the approximating system and we prove the truncated formulation of the momentum equation. In the Section \ref{sec:main} we prove Theorem \ref{teo:main}.

\section{Preliminaries} \label{sec:pre}
\subsection{Notations}
Let $\T$ the three-dimensional torus $[0,1]^3$, the space of smooth functions with value in $\R^d$ compactly supported in the box $[0,1]^3$ and extended periodically will be $C^{\infty}_c((0,T)\times\T;\R^{d})$. We will denote with $L^{p}(\T)$ the standard Lebesgue spaces and with $\|\cdot\|_{L^p}$ their norm. The Sobolev space of functions with $k$ distributional derivatives in $L^{p}(\T)$ is $W^{k,p}(\T)$ and in the case $p=2$ we will write $H^{k}(\T)$. The spaces $W^{-k,p}(\T)$ and $H^{-k}(\T)$ denote the dual spaces of $W^{k,p'}(\T)$ and $H^{k}(\T)$ where  $p'$ is the H\"older conjugate of $p$. Given a Banach space $X$ we use the classical Bochner space for time dependent functions with value in $X$, namely $L^{p}(0,T;X)$, $W^{k,p}(0,T;X)$ and $W^{-k,p}(0,T;X)$ and when $X=L^p(\Omega)$, the norm of the space $L^{q}(0,T;L^{p}(\Omega))$ is denoted by $\|\cdot\|_{L^{q}_{t}L^{p}_{x}}$. Then, the space $C(0,T;X_w)$ is the space of continuous functions with value in the space $X$ endowed with the weak topology. The space of Radon measure defined by the duality with continuous compactly supported functions will be denoted by $\mathcal{M}(\T;\R^{d})$. Next, we denote by $Du=(\nabla u+(\nabla u)^T)/2$ the symmetric part of the gradient and by $Au=(\nabla u-(\nabla u)^T)/2$ the antisymmetric one. Given a matrix  $C\in\R^{3\times 3}$ we denote by $C^{s}$, the symmetric part of $C$ and by $C^{a}$ the antisymmetric one.  
\subsection{Definition of weak solutions and statement of the main result}\label{subsec:main}
The definition of weak solution for the system \eqref{eq:qns1}-\eqref{eq:qns2} is the following
 \begin{definition}\label{def:ws}
A triple $(\rho, u, \MT )$ with $\rho\geq0$ is said to be a weak solution of the Cauchy problem \eqref{eq:qns1}-\eqref{eq:qns2}-\eqref{eq:id} if the following conditions are satisfied. 
\begin{enumerate}
\item Integrability conditions.
\begin{align*}
&\rho\in L^{\infty}(0,T;H^{1}(\T))\cap L^{2}(0,T;H^2(\T)),
&\rrho\,u\in L^{\infty}(0,T;L^{2}(\T)),\\
&\rho^{\frac{\gamma}{2}}\in L^{\infty}(0,T;L^2(\T))\cap L^{2}(0,T;H^1(\T)),
&\nabla\rrho\in L^{\infty}(0,T;L^{2}(\T)).\\
&\MT\in L^{2}(0,T;L^{2}(\T)),&\rho\,u\in C([0,T);L^{\frac{3}{2}}_{w}(\T)).
\end{align*}
\item Continuity equation:\\
For any $\phi\in C_c^{\infty}([0,T)\times\T;\R)$
\begin{equation}\label{eq:wfce}
\int\rho^0\phi(0)\,dx+\iint\rho\phi_t+\rrho\rrho u\nabla\phi\,dxdt=0.
\end{equation}
\item Momentum equation:\\
For any fixed $l=1,2,3$ and $\psi\in  C_c^{\infty}([0,T)\times\T;\R^3)$
\begin{equation}\label{eq:wfme}
\begin{aligned}
&\int \rho^0u^{0,l}\psi(0)\,dx+\iint\rrho(\rrho u^{l})\psi_t\,dxdt+\iint\rrho u\rrho u^{l}:\nabla\psi\,dxdt\\
&-\iint\rrho\MT^{s}_{\cdot,l}\nabla\psi-2\iint\nabla\rho^{\frac{\gamma}{2}}\rho^{\frac{\gamma}{2}}\cdot\psi\,dxdt-\iint\nabla_{l}\rho\Delta\rho\psi\,dxdt\\
&-\iint\rho\Delta\rho\nabla_{l}\psi\,dxdt=0.
\end{aligned}
\end{equation}
\item Dissipation:\\
For any $\varphi\in C_c^{\infty}([0,T)\times\T;\R)$
\begin{equation}\label{eq:dissor}
\iint\rrho\MT_{i,j}\varphi\,dxdt=-\iint\rho u_{i}\nabla_{j}\varphi\,dxdt-\iint2\rrho u_i\otimes\nabla_{j}\rrho\,\varphi\,dxdt.
\end{equation}
 \item Energy Inequality:\\
There exist $\Lambda\in L^{\infty}_{t}(L^{2}_{x})$, with $\rho\,u=\rrho\Lambda$ a.e. in $(0,T)\times\T$, such that the following energy inequality hold
\begin{equation*}
\begin{aligned}
&\sup_{t\in(0,T)}\int_{\T}\frac{|\Lambda(t,x)|^2}{2}+\frac{\rho(t,x)^{\gamma}}{\gamma-1}+\frac{|\nabla\rho(t,x)|^2}{2}\,dx+\iint|\MT^{s}(t,x)|^2\,dxdt\\
&\leq\int_{\T}\rho^0(x)|u^0(x)|^2+\frac{{\rho^{0}(x)}^{\gamma}}{\gamma-1}+\frac{|\nabla\rho^0(x)|^2}{2}\,dx.
\end{aligned}
\end{equation*}
\end{enumerate}
\end{definition}
In order to state our main result, we first specify the assumptions on the initial data. We assume that \begin{equation}\label{eq:hyidr}
\begin{aligned}
&\rho^0\geq0,\quad&\rho^0\in L^{1}\cap L^{\gamma}(\T),\quad&\nabla\sqrt{\rho^0}\in L^{2}(\T),\quad\log\rho^0\in L^{1}(\T).\\
\end{aligned}
\end{equation}
We point out that the assumption on the summability of $\log\rho^0$ is made only to avoid the technicalities in approximating the initial data. Regarding the initial velocity, we assume that $u^0$ is a measurable vector field, finite almost everywhere such that \begin{equation}\label{eq:hyidu}
\begin{aligned}
&\sqrt{\rho^0}u^0\in L^{2}(\T),\quad
&\rho^0 u^0\in L^{p}(\R^3)\textrm{ with }p<2.\\ 
\end{aligned}
\end{equation}

The main theorem of this paper is the following. 
\begin{theorem}\label{teo:main}
Assume $\rho^0$ and $\rho^0\,u^0$ satisfy \eqref{eq:hyidr} and \eqref{eq:hyidu}. Then, there exist at least a weak solution $(\rho, u, \MT)$ of \eqref{eq:qns1}-\eqref{eq:id} in the sense of Definition \ref{def:ws}.
\end{theorem}
\begin{remark}
We stress that the velocity field is not uniquely defined on the vacuum region $\{\rho=0\}$. Moreover, we are not able to deduce that $\Lambda=\rrho u$ in Definition \ref{def:ws}, since we do not know if $\Lambda=0$ on $\{\rho=0\}$. 
\end{remark}
\subsection{The Truncations}\label{subsec:truncations}
Let $\bar{\beta}:\R\to\R$ be an even positive compactly supported smooth function such that 
\begin{equation*}
\bar{\beta}(z)=1\textrm{ for }z\in[-1,1],
\end{equation*}
$\mbox{supp}\,\bar{\beta}\subset (-2,2)$ and $0\leq \bar{\beta}\leq 1$. Given $\bar{\beta}$, we define 
$\tilde{\beta}:\R\to\R$ as follows: 
\begin{equation*}
\tilde{\beta}(z)=\int_{0}^{z}\bar{\beta}(s)\,ds. 
\end{equation*}
For $y\in\R^{3}$ we define for any $\delta>0$ the functions 
\begin{equation*}
\begin{aligned}
&\beta_{\delta}^{1}(y):=\frac{1}{\delta}\tilde{\beta}(\delta\, y_1)\bar{\beta}(\delta\, y_2)\bar{\beta}(\delta\, y_3),\\
&\beta_{\delta}^{2}(y):=\frac{1}{\delta}\bar{\beta}(\delta\, y_1)\tilde{\beta}(\delta\, y_2)\bar{\beta}(\delta\, y_3),\\
&\beta_{\delta}^{3}(y):=\frac{1}{\delta}\bar{\beta}(\delta\, y_1)\bar{\beta}(\delta\, y_2)\tilde{\beta}(\delta\, y_3).
\end{aligned}
\end{equation*}
Note that for fixed $l=1,2,3$ the function $\beta_{\delta}^{l}:\R^{3}\to\R$ is a truncation of the function $f(y)=y_l$. 
Finally, for any $\delta>0$  we define $\hat{\beta}_{\delta}:\R^{3}\to\R$ as
\begin{equation*}
\hat{\beta}_{\delta}(y):=\bar{\beta}(\delta\,y_1)\bar{\beta}(\delta\,y_2)\bar{\beta}(\delta\,y_3),
\end{equation*}
and for any $\la>0$ we define $\bar{\beta}_{\lambda}:\R\to\R$ as
$$\bar{\beta}_{\lambda}(s)=\bar{\beta}(\lambda\,s).$$
In the next Lemma we collect some of the main properties of $\beta_{\delta}^{l}$, $\hat{\beta}_{\delta}$ and $\bar{\beta}_{\lambda}$. Those properties are elementary and can be deduced directly from the definitions.
\begin{lemma}\label{lem:trunc}
Let $\lambda,\,\delta>0$ and $K:=\|\bar{\beta}\|_{W^{2,\infty}}$. Then, there exists $C=C(K)$ such that the following bounds hold. 
\begin{enumerate}
\item For any $\delta>0$ and $l=1,2,3$
\begin{equation}\label{eq:bed}
\begin{aligned}
&\|\beta^{l}_{\delta}\|_{L^{\infty}}\leq \frac{C}{\delta},\quad&\|\nabla\beta^{l}_{\delta}\|_{L^{\infty}}\leq{C},\quad
&\|\nabla^{2}\beta^{l}_{\delta}\|_{L^{\infty}}\leq C\,\delta,
\end{aligned}
\end{equation}
\item For any $\la>0$
\begin{equation}\label{eq:bel}
\begin{aligned}
&\|\bar{\beta}_{\la}\|_{L^{\infty}}\leq 1,\quad&\|\bar{\beta}'_{\la}\|_{L^{\infty}}\leq C\,\la,\quad&\sqrt{|s|}\bar{\beta}_{\la}(s)\leq \frac{C}{\sqrt{\la}}.
\end{aligned}
\end{equation}
\item For any $\delta>0$
\begin{equation}\label{eq:ybed}
\begin{aligned}
&&\|\hat{\beta}_{\delta}\|_{L^{\infty}}\leq 1,\quad&\|\nabla\hat{\beta}_{\delta}\|_{L^{\infty}}\leq{C\delta},\quad&|y||\hat{\beta}_{\delta}(y)|\leq \frac{C}{\delta},
\end{aligned}
\end{equation}
\item The following convergences hold for $l=1,2,3$, pointwise on $\R^{3}$, as $\delta\to 0$
\begin{equation}\label{eq:bedc}
\begin{aligned}
&\beta_{\delta}^{l}(y)\to y_{l},\quad
&(\nabla_{y}\beta_{\delta}^{l})(y)\to \nabla_{y_{l}} y,\quad
&\hat{\beta}_{\delta}(y)\to1.
\end{aligned}
\end{equation}
\item The following convergence holds pointwise on $\R$ as $\la\to 0$
\begin{equation}\label{eq:belc}
\begin{aligned}
&\bar{\beta}_{\la}(s)\to 1.
\end{aligned}
\end{equation}
\end{enumerate}
\end{lemma}
\subsection{DiPerna-Lions Commutator Estimate}
In this Subsection we recall the commutator estimate for convolutions of DiPerna-Lions \cite{DPL}. 
First, for any function $f$ we denote by $\overline{f}_{r}$ the time-space convolution of $f$ with a smooth sequence of even mollifiers $\{\Psi_{r}\}_{r}$, namely 
\begin{equation*}
\overline{f}_{r}=\Psi_r*f(t,x),\quad t>r, 
\end{equation*}
where 
\begin{equation*}
\Psi_r(t,x)=\frac{1}{r^{4}}\Psi(\frac{t}{r},\frac{x}{r})
\end{equation*}
and $\Psi$ is a smooth nonnegative even function such that $\mbox{supp}\Psi\subset B_1(0)$ and 
$$
\iint\Psi\,dxdt=1.$$
Then, the following lemma hold true. 
\begin{lemma}\label{lem:DPL}
Let $p_1,\,p_2\in [1,\infty]$ and $p_3<\infty$.   
\begin{enumerate}
\item Let $B\in L^{p_1}((0,T)\times\T;\R^3)$ such that $\nabla\,B\in L^{p_1}((0,T)\times\T;\R)$ and let $f\in L^{p_2}((0,T)\times\T;\R)$, then 
\begin{equation*}
\begin{aligned}
\|\dive\overline{(B\,f)_r}-\dive(B\overline{f}_{r})\|_{L^{p_3}_{t,x}}\to 0\mbox{ as }r\to0
\end{aligned}
\end{equation*}
provided $\frac{1}{p_3}=\frac{1}{p_1}+\frac{1}{p_2}$. 
\item Let $g\in L^{p_1}((0,T)\times\T;\R)$ such that $\partial_t\,g\in L^{p_1}((0,T)\times\T;\R)$ and let $f\in L^{p_2}((0,T)\times\T;\R)$, then 
\begin{equation*}
\begin{aligned}
\|\overline{\partial_t(g\,f)_r}-\partial_t(g\overline{f}_{r})\|_{L^{p_3}_{t,x}}\to 0\mbox{ as }r\to0, 
\end{aligned}
\end{equation*}
provided $\frac{1}{p_3}=\frac{1}{p_1}+\frac{1}{p_2}$. 
\end{enumerate}
\end{lemma}
We omit the proof of the Lemma. Notice that part 1) can be easily deduced from \cite[Lemma II.1]{DPL} and part 2) is a simple corollary. 

\section{Weak solutions of approximating system and their properties}\label{sec:app}
We will construct weak solutions of \eqref{eq:qns1}-\eqref{eq:qns2} as limit of weak solutions of the following system:
\begin{equation}\label{eq:app}
\begin{aligned}
&\partial_t\re+\dive(\re\ue)=0,\\
&\partial_{t}(\re\ue)+\dive(\re\ue\otimes\ue)-\dive(\re\,D\ue)+\nabla\re^{\frac{\gamma}{2}}+\e\re|\ue|^2\ue+\e\ue\\&=\re\nabla\Delta\re+\e\re\nabla\left(\frac{\Delta\sqrt{\re}}{\sqrt{\re}}\right),
\end{aligned}
\end{equation}
with initial datum 
\begin{equation}\label{eq:idapp}
\begin{aligned}
\rn(0,x)&=\rho^0(x),\\
(\rn\un)(0,x)&=\rho^0(x)u^0(x),
\end{aligned}
\end{equation} 
satisfying the hypothesis \eqref{eq:hyidr} and \eqref{eq:hyidu}. The definition of weak solutions of the system \eqref{eq:app} is the following. 
\begin{definition}\label{def:aws}
A triple $(\rho_{\e}, u_{\e}, \MT_{\e})$ is a weak solution of \eqref{eq:app}-\eqref{eq:idapp} provided the following properties hold. 
\begin{enumerate}
\item Integrability Hypothesis:
\begin{align*}
&\re\in L^{\infty}(0,T;H^{1}(\T))\cap L^{2}(0,T;H^2(\T)),
&\re\,\ue\in L^{\infty}(0,T;L^{2}(\T)),\\
&\re^{\frac{\gamma}{2}}\in L^{\infty}(0,T;L^2(\T))\cap L^{2}(0,T;H^1(\T)),
&\MT_{\e}\in L^{\infty}(0,T;L^{2}(\T)),\\
&\rre\in L^{\infty}(0,T;H^{1}(\T))\cap L^{2}(0,T;H^2(\T)),&\re^{\frac{1}{4}}\ue\in L^{4}((0,T)\times(\T)),\\
&\ue\in L^{2}((0,T)\times(\T)),&\rn\ue\in C([0,T);L^{\frac{3}{2}}_{w}(\T)). 
\end{align*}
\item Continuity equation:\\
For any $\phi\in C^{\infty}_{c}([0,T)\times\T;\R)$
\begin{equation}\label{eq:exce}
\iint\re\partial_t\phi+\re\ue\nabla\phi\,dxdt+\int\rho^{0}\phi(0)\,dx=0.
\end{equation}
\item Momentum equation:\\
For any $\psi\in C^{\infty}_{c}([0,T)\times\T;\R^{3})$
\begin{equation}\label{eq:exme}
\begin{aligned}
&\iint\re\ue\partial_t\psi+\re\ue\otimes\ue\nabla\psi-\rre\MT^{s}_{\e}\nabla\psi-\e\re|\ue|^{2}\ue\psi-\e\ue\psi\,dxdt\\
&\iint-2\re^{\frac{\gamma}{2}}\nabla\re^{\frac{\gamma}{2}}\psi-\e\rre\nabla^{2}\rre\nabla\psi+\e\nabla\rre\otimes\nabla\rre\nabla\psi\,dxdt\\
&\iint-\nabla\re\Delta\re\psi\re\Delta\re\nabla\psi\,dxdt+\int\rho^0u^0\psi(0)\,dx.
\end{aligned}
\end{equation}
\item Dissipation:\\
For any $\varphi\in C^{\infty}_{c}([0,T)\times\T;\R)$
\begin{equation}\label{eq:diss}
\iint\rre\MT_{\e,i,j}\varphi\,dxdt=-\iint\re\,u_{\e,i}\,\nabla_j\varphi\,dxdt-2\iint\rre\,u_{\e,i}\otimes\nabla_i\rre\varphi\,dxdt.
\end{equation}
\item Energy Estimate:
\begin{equation}\label{eq:en}
\begin{aligned}
\sup_{t\in (0,T)}&\left(\int\rn\frac{|\un|^2}{2}+\frac{\rn^{\gamma}}{\gamma-1}+\frac{|\nabla\rn|^2}{2}+\e|\nabla\rre|^2\,dx\right)\\
&+\iint|\MT^{s}_{\e}|^2\,dxdt +\e\iint\re|\ue|^4\,dxdt+\e\iint|\ue|^2\,dxdt\\
&\leq\int\rho^0\frac{|u^0|^2}{2}+\frac{{\rho^0}^{\gamma}}{\gamma-1}+\frac{|\nabla\rho^0|^2}{2}+\e|\nabla\rrho^0|^2\,dx.
\end{aligned}
\end{equation}
\item BD Entropy
\begin{equation}\label{eq:bd1}
\begin{aligned}
\sup_{t\in(0,T)}&\left(\int\rn\frac{|\wn|^2}{2}+\frac{\rn^\gamma}{\gamma-1}+\frac{|\nabla\rn|^2}{2}+(\re-\e\log\,\re)+\e|\nabla\rre|^{2}\,dx\right)\\
+\frac{4}{\gamma}&\iint|\nabla\rn^{\frac{\gamma}{2}}|^2\,dxdt
+\frac{1}{2}\iint\MT^{a}_{\e}\,dxdt+\iint|\Delta \rn|^2\,dxdt\\
\frac{\e}{C}&\iint|\nabla^2\rre|^2+|\nabla\re^{\frac{1}{4}}|^{4}\,dxdt+\e\iint\re|\ue|^4+|\ue|^2\,dxdt\\
&\leq\int\rho^0\frac{|u^0|^2}{2}+\frac{2{\rho^0}^{\gamma}}{\gamma-1}+\frac{2|\nabla\rho^0|^2}{2}+2\e|\nabla\rrho^0|^2\,dx\int\rho^0\frac{|w^0|^2}{2}\\
&+\int(\rho^0-\e\log\,\rho^0)\,dx.
\end{aligned}
\end{equation}
\end{enumerate}
\end{definition} 
Following the arguments in \cite{VY2} ad \cite{BD} it is easy to prove that there exists at least a weak solution in the sense of the Definition \ref{def:aws} so we omit the proof. In this section we prove the truncated formulation of the momentum equation.  
\begin{theorem}\label{teo:exmain}
Let $(\re,\ue,\MT_{\e})$ be a weak solution of the system \eqref{eq:app}-\eqref{eq:idapp} in the sense of Definition \ref{def:aws}. Let $\beta^{l}_{\delta}$ and $\bar{\beta}_{\lambda}$ the truncation defined in the Subsection \ref{subsec:truncations}. Then the following equalities hold.
\begin{enumerate}
\item For any $\psi\in C^{\infty}_{c}([0,T)\times\T;\R)$ 
\begin{equation}\label{eq:ren211}
\begin{aligned}
&\int\rho^{0}\beta_{\delta}^{l}(u^0)\bar{\beta}_{\la}(\rho^0)\psi(0,x)\,dx+\iint\re\bedn\beln\partial_t\psi-\iint\rn \ue\bedn\beln\cdot\nabla\psi\,dxdt\\
&-\iint\rre\MT^{s}_{\e}:\nablay\bedn\beln\otimes\nabla\psi\,dxdt-2\iint\rn^{\frac{\gamma}{2}}\nabla\rn^{\frac{\gamma}{2}}\cdot\nabla_y\bedn\beln\psi\,dxdt\\
&-\iint\nabla\rn\Delta\rn\nablay\bedn\beln\psi\,dxdt-\iint\rn\Delta\rn\nablay\bedn\beln\nabla\psi\,dxdt\\
&+\iint R^{\delta,\la}_{\e}\psi\,dxdt+\iint \tilde{R}^{\delta,\la}_{\e}\psi\,dxdt=0.
\end{aligned}
\end{equation}
where the remainders are
\begin{equation}\label{eq:remainder1}
\begin{aligned}
R^{\delta,\la}_{\e}=\sum_{i=1}^{6}R^{\delta,\la}_{\e,i}&=\rn\bedn\belnp\partial_t\rn+\rn u\bedn\belnp\nabla\rn\\
&-\rre\MT^{s}_{\e}:\nabla_{y}\bedn\otimes\nabla\rn\belnp+\rre\Delta\re\nabla_{y}^2\bedn\MT_{\e}\beln\\
\\
&+\rn\Delta\rn\nablay\bedn\belnp\nabla\rn-\MT^{s}_{\e}\MT_{\e}\nabla_{y}^2\bedn\beln.
\end{aligned}
\end{equation}
\begin{equation}\label{eq:remainder1bis}
\begin{aligned}
\tilde{R}^{\delta,\la}_{\e}=\sum_{i=1}^{6}\tilde{R}^{\delta,\la}_{\e,i}&=-\e\nabla^{2}\rrn\,\MT_{\e}\,\nabla^{2}_{y}\bedn\beln+4\e\nabla\rn^{\frac{1}{4}}\otimes\nabla\rn^{\frac{1}{4}}\,\MT_{\e}\nabla^{2}_{y}\bedn\beln\\
\\
&-\e\rrn\nabla^{2}\rrn\nabla_{y}\bedn\belnp\nabla\rn+4\e\rrn\nabla\rn^{\frac{1}{4}}\otimes\nabla\rn^{\frac{1}{4}}\bedn\belnp\nabla\rn\\
\\
&-\e\rho|u|^{2}u\nabla_{y}\bed\bel-\e u\nabla_{y}\bed\bel
\end{aligned}
\end{equation}
\bigskip

\item For any $\varphi\in C^{\infty}_{c}((0,T)\times\T;\R)$ the following (tensor) equality holds:
\begin{equation}\label{eq:dissren}
\begin{aligned}
\iint \rrn\MT_{\e}\hat{\beta}_{\delta}(\ue)\varphi\,dxdt=&-\iint \hat{\beta}_{\delta}(\ue)\rn \ue\otimes\nabla\varphi\,dxdt-\iint\rrn \ue\varphi\nablay\hat{\beta}_{\delta}(\ue)\MT_{\e}\,dxdt\\
&-2\iint\rrn \ue\otimes \nabla\rre\varphi\hat{\beta}_{\delta}(\ue)\,dxdt.
\end{aligned}
\end{equation}
\end{enumerate}
\end{theorem}
\begin{proof}\mbox{}
In order to avoid heavy notations in the following theorem and the proof we drop the subscript $\e$ in $(\re,\ue,\MT_{\e})$. 
\bigskip

\noindent 1) First, we define the following quantity: 
\begin{equation*}
\begin{aligned}
&M:=\rrho\MT^{s}+\e\rrho\nabla^{2}\rrho-\e\nabla\rrho\otimes\nabla\rrho,\\
&N:=\e\rho|u|^{2}u+\e u+2\rho^{\frac{\gamma}{2}}\nabla\rho^{\frac{\gamma}{2}}.
\end{aligned}
\end{equation*}
Moreover, let $\phi_m(y)$ be the function defined as follows 
\begin{equation*}
\phi_m(y)=\left\{
\begin{array}{rll}
&0 & \mbox{for $0< y\leq \frac{1}{2m}$ }, \\
\\
&2my-1 & \mbox{for $\frac{1}{2m}\leq y\leq \frac{1}{m}$ },\\
\\
&1 & \mbox{for $\frac{1}{m}\leq y\leq m$ }, \\
\\
&2-\frac{y}{m} & \mbox{for $m\leq y\leq 2m$ }, \\
\\
&0 & \mbox{for $2m\leq y$}.
\end{array}
\right.
\end{equation*}
For fixed $\e$, by using the energy estimate and the BD entropy, we have the following bounds
\begin{equation}\label{eq:exint}
\begin{aligned}
&\|\rho\|_{L^{\infty}_t(L_x^{1}\cap L^{\gamma}_x)}\leq C,\quad\|\nabla\rho\|_{L_t^{\infty}L_x^{2}}\leq C,\quad\|\nabla\rrho\|_{L_t^{\infty}L_x^{2}}\leq C, \\
&\|\nabla\rho^{\frac{\gamma}{2}}\|_{L^{2}_{t,x}}\leq C,\quad\|\e^{\frac{1}{2}}\nabla^{2}\rrho\|_{L^{2}_{t,x}}\leq C,\quad\|\e^{\frac{1}{4}}\nabla\rho^{\frac{1}{4}}\|_{L^{4}_{t,x}}\leq C,\quad \|\e^{\frac{1}{4}}\rho^{\frac{1}{4}}u\|_{L^{4}_{t,x}}\leq C,\\
&\|\rrho\,u\|_{L_t^{\infty}L_x^{2}}\leq C,\quad \|\MT\|_{L^{2}_{t,x}}\leq C,\quad\|\sqrt{\e}u\|_{L^{2}_{t,x}}\leq C,\quad\|\nabla^2\rho\|_{L^{2}_{t,x}}\leq C,
\end{aligned}
\end{equation}
where $C$ depends only on the initial data \eqref{eq:id}.
 Notice that, by using a combination of bounds in \eqref{eq:exint} and standard interpolations inequalities the following bounds hold true for fixed $\e$
 \begin{equation}\label{eq:p5}
\begin{aligned}
&\|\partial_t\rho\|_{L^{\frac{4}{3}}_{t,x}}\leq C_{\e},\quad&\|\nabla(\rho u)\|_{L^{\frac{4}{3}}_{t,x}}\leq C_{\e},\quad
&\|\nabla\rho\Delta\rho\|_{L^{\frac{5}{4}}_{t,x}}\leq C_\e,\quad
&\|\rho\Delta\rho\|_{L^{\frac{4}{3}}_{t,x}}\leq C_{\e}.
\end{aligned}
\end{equation}
Moreover, by \eqref{eq:p5} and the definition of $\phi_{m}$, for any fixed $m\in\N$ and any foxed $\e$ we have that 
\begin{equation}\label{eq:p5bis}
\begin{aligned}
&\|\phi'_{m}(\rho)M\|_{L^{2}_{t,x}}\leq C_{m,\e},\quad&\|\phi_{m}(\rho)N\|_{L^{\frac{4}{3}}_{t,x}}\leq C_{m,\e},\\
&\|\rho\Delta\rho\phi'_m(\rho)\nabla\rho\|_{L_{t,x}^{\frac{5}{4}}}\leq C_{m,\e},\quad&\|\rho\Delta\rho\phi_m(\rho)\|_{L^{2}_{t,x}}\leq C_{m,\e},\\
 &\|\phi_{m}(\rho)M\|_{L^{2}_{t,x}}\leq C_{m,\e}\,&\|\phi'(\rho)\nabla\rho\|_{L^{4}_{t,x}}\leq C_{m,\e}.
\end{aligned}
\end{equation}
Finally, define $v_m:=\phi_m(\rho)u$. By using \eqref{eq:exint} and the definition of $\phi_m$ we have that 
\begin{equation}\label{eq:estvm}
\begin{aligned}
&\|v_m\|_{L^{4}_{t,x}}\leq C_{m,\e}\, &\|\nabla v_m\|_{L^{2}_{t,x}}\leq C_{m,\e}.
\end{aligned}
\end{equation}
Consider the weak formulation of the momentum equation in Definition \ref{def:aws}, namely
\begin{equation}\label{eq:pe1}
\begin{aligned}
&\iint\rho u\partial_t\psi+\rho u\otimes u\nabla\psi-M\nabla\psi-N\psi-\nabla\rho\Delta\rho\psi\,dxdt\\
&-\iint\rho\Delta\rho\nabla\psi\,dxdt+\int\rho^0 u^0\psi(0)\,dx=0,
\end{aligned}
\end{equation}
with $\psi\in C^{\infty}_{c}((0,T)\times\T;\R)$. 
Then, the following equation for $v_{m}$ holds 
\begin{equation}\label{eq:vm}
\begin{aligned}
&\iint\rho v_m\partial_t\psi+\rho u\otimes v_m\nabla\psi-\psi\rrho\textrm{tr}(\MT)\phi'_m(\rho)\rho u\,dxdt\\
-&\iint\phi_{m}(\rho)\,M\,\nabla\psi\,dxdt-\iint M\phi'_{m}(\rho)\nabla \rho\psi\,dxdt-\iint N\phi_{m}(\rho)\psi\,dxdt\\
+&\iint\nabla\rho\Delta\rho\phi_m(\rho)\psi\,dxdt+\iint\rho\Delta\rho\phi'_m(\rho)\nabla\rho\psi\,dxdt+\iint\rho\Delta\rho\phi_m(\rho)\nabla\psi\,dxdt=0.
\end{aligned}
\end{equation}
For a full derivation of \eqref{eq:vm} we refer to \cite[Section 3.2]{LV}. The equality \eqref{eq:vm} is obtained by using $\overline{\phi_m(\rho)\psi}_{r}$ as test function in \eqref{eq:exme} and using correctly the continuity equation. By using standard properties of mollifiers and an easy combination of the bounds \eqref{eq:exint}-\eqref{eq:p5bis}, it is possible to send $r\to 0$ and the equality \eqref{eq:vm} follows easily. Although it is a very standard procedure, we show how to deal with the capillarity terms. 
Note that 
\begin{equation}\label{eq:p3}
\begin{aligned}
&\iint\nabla\rho\Delta\rho\,\overline{\phi_m(\rho)\psi}_{r}+\rho\Delta\rho\nabla[\overline{\phi_m(\rho)\psi}_{r}]\,dxdt\\
&=\iint \overline{\nabla\rho\Delta\rho}_{r}\phi_m(\rho)\psi+\overline{\rho\Delta\rho}_r\phi'_m(\rho)\nabla\rho\psi
+\overline{\rho\Delta\rho}_r\phi_m(\rho)\psi\,dxdt. 
\end{aligned}
\end{equation}
Then, by using the bounds in \eqref{eq:p5} it follows that 
\begin{equation*}
\begin{aligned}
&\iint\overline{\nabla\rho\Delta\rho}_{r}\phi_m(\rho)\psi\,dxdt\to \iint\nabla\rho\Delta\rho\phi_m(\rho)\psi\,dxdt,\\
&\iint\overline{\rho\Delta\rho}_r\phi'_m(\rho)\nabla\rho\psi\,dxdt\to \iint\rho\Delta\rho\phi'_m(\rho)\nabla\rho\psi\,dxdt,\\
&\iint\overline{\rho\Delta\rho}_r\phi_m(\rho)\psi\,dxdt\to \iint\rho\Delta\rho\phi_m(\rho)\psi\,dxdt.
 \end{aligned}
 \end{equation*}
 Next, we move to the proof of the truncated formulation of the momentum equation. Consider as test function 
 $\overline{\nabla_{y}\beta^{l}_{\delta}(\overline{v_m}_r)\psi}_r$ in \eqref{eq:vm} with $\psi\in C^{\infty}_{c}((0,T)\times\T;\R)$. By sending $r\to 0$ we get 
 \begin{equation}\label{eq:ren}
 \begin{aligned}
&\iint\rho \bedm\partial_t\psi+\rho u\bedm\nabla\psi-\rrho\textrm{tr}(\MT)\phi'_m(\rho)\rho u\nabla_y\bedm\psi\,dxdt\\
&-\iint\phi_{m}(\rho)\,M\,\nabla_y^2\bedm\nabla v_m\psi\,dxdt-\iint\phi_{m}(\rho)\,M\,\nabla_y\bedm\nabla\psi\,dxdt\\
&-\iint M\phi'_{m}(\rho)\nabla \rho\nabla_y\bedm\psi\,dxdt-\iint N\phi_{m}(\rho)\nabla_y\bedm\psi\,dxdt\\
&+\iint\nabla\rho\Delta\rho\phi_m(\rho)\nabla_y\bedm\psi\,dxdt+\iint\rho\Delta\rho\phi'_m(\rho)\nabla\rho\nabla_y\bedm\psi\,dxdt\\
&+\iint\rho\Delta\rho\phi_m(\rho)\nabla_y^2\bedm\nabla v_m\psi\,dxdt+\iint\rho\Delta\rho\phi_m(\rho)\nabla_y\bedm\nabla\psi\,dxdt=0.
\end{aligned}
\end{equation}
As before, a full derivation of the previous equality can be found in \cite[Section 3.3]{LV},
here we give the details on how to treat the transport term and the capillarity terms. 
Regarding the transport term, by using standard properties of the convolutions we have that 
\begin{equation*}
\begin{aligned}
&-\iint\rho v_m\partial_t\left[\overline{\nabla_{y}\beta^{l}_{\delta}(\overline{v_m}_r)\psi}_r\right]+\rho u\otimes v_m\nabla\left[\overline{\nabla_{y}\beta^{l}_{\delta}(\overline{v_m}_r)\psi}_r\right]\,dxdt\\
&=\iint\left[\partial_t(\overline{\rho v_m}_{r})+\dive(\overline{\rho u\otimes v_m}_r)\right]\nabla_{y}\beta^{l}_{\delta}(\overline{v_m}_r)\psi\,dxdt\\
&=\iint\left[\partial_t(\overline{\rho v_m}_{r})-\partial_t(\rho \overline{v_m}_{r})+\dive(\overline{\rho u\otimes v_m}_r)-\dive(\rho u\otimes \overline{v_m}_r)\right]\nabla_{y}\beta^{l}_{\delta}(\overline{v_m}_r)\psi\,dxdt\\
&+\iint\left[\partial_t(\rho \overline{v_m}_{r})+\dive(\rho u\otimes \overline{v_m}_r)\right]\nabla_{y}\beta^{l}_{\delta}(\overline{v_m}_r)\psi\,dxdt=I_r+II_r
\end{aligned}
\end{equation*}
The term $I_r$ is treated by using the commutator estimate of DiPerna-Lions \cite{DPL}. Indeed, by Lemma \ref{lem:DPL} part 2) with $g=\rho$ and $f=v_{i,m}$ and Lemma \ref{lem:DPL} with $B=\rho\,u$ and $f=v_{i,m}$, by using the bounds in \eqref{eq:p5} and \eqref{eq:estvm} we have that 
\begin{equation*}
I_r\to 0\textrm{ as }r\to 0. 
\end{equation*}
Regarding the second term, we first note that by \eqref{eq:p5} the continuity equation holds a.e. in $(0,T)\times\T$, then, since $\overline{v_m}_r$ is smooth, we have that 
\begin{equation*}
\begin{aligned}
II_r&=\iint\left(\partial_t(\rho \beta^{l}_{\delta}(\overline{v_m}_{r}))+\dive(\rho u\beta^{l}_{\delta} \overline{v_m}_r))\right)\psi\,dxdt\\
&-\iint\left(\rho \beta^{l}_{\delta}(\overline{v_m}_{r})\partial_t\psi+\rho u\beta^{l}_{\delta}(\overline{v_m}_r)\nabla\psi\right)\,dxdt,
\end{aligned}
\end{equation*}
which by using standard properties of convolutions and \eqref{eq:p5} converges to 
\begin{equation*}
-\iint\left(\rho \beta^{l}_{\delta}(v_m)\partial_t\psi+\rho u\beta^{l}_{\delta}(v_m)\nabla\psi\right)\,dxdt.
\end{equation*}
Regarding the capillarity terms, by standard properties of convolutions we have that  
\begin{equation*}
\begin{aligned}
&\iint\nabla\rho\Delta\rho\phi_m(\rho)\overline{\nabla_{y}\beta^{l}_{\delta}(\overline{v_m}_r)\psi}_r\,dxdt+\iint\rho\Delta\rho\phi'_m(\rho)\nabla\rho\overline{\nabla_{y}\beta^{l}_{\delta}(\overline{v_m}_r)\psi}_r\,dxdt\\
&+\iint\rho\Delta\rho\phi_m(\rho)\nabla[\overline{\nabla_{y}\beta^{l}_{\delta}(\overline{v_m}_r)\psi}_r]\,dxdt
=\iint\overline{\nabla\rho\Delta\rho\phi_m(\rho)}_{r}\nabla_{y}\beta^{l}_{\delta}(\overline{v_m}_r)\psi\,dxdt\\
&+\iint\overline{\rho\Delta\rho\phi'_m(\rho)\nabla\rho}_r\nabla_{y}\beta^{l}_{\delta}(\overline{v_m}_r)\psi\,dxdt
+\iint\overline{\rho\Delta\rho\phi_m(\rho)}_{r}\nabla^2_{y}\beta^{l}_{\delta}(\overline{v_m}_r)\nabla\overline{v_m}_r\psi\,dxdt\\
&+\iint\overline{\rho\Delta\rho\phi_m(\rho)}_{r}\nabla_{y}\beta^{l}_{\delta}(\overline{v_m}_r)\nabla\psi\,dxdt.
\end{aligned}
 \end{equation*}
Then, by using \eqref{eq:p5}-\eqref{eq:estvm} and the fact that $\beta^{l}_{\delta}\in W^{2,\infty}(\R)$, the following convergences as $r\to 0$ follow easily 
 \begin{equation*}
 \begin{aligned}
&\iint\overline{\nabla\rho\Delta\rho\phi_m(\rho)}_{r}\nabla_{y}\beta^{l}_{\delta}(\overline{v_m}_r)\psi\,dxdt\to\iint\nabla\rho\Delta\rho\phi_m(\rho)\nabla_{y}\beta^{l}_{\delta}(v_m)\psi\,dxdt,\\
&\iint\overline{\rho\Delta\rho\phi'_m(\rho)\nabla\rho}_r\nabla_{y}\beta^{l}_{\delta}(\overline{v_m}_r)\psi\,dxdt\to
\iint\rho\Delta\rho\phi'_m(\rho)\nabla\rho\nabla_{y}\beta^{l}_{\delta}(v_m)\psi\,dxdt,\\
&\iint\overline{\rho\Delta\rho\phi_m(\rho)}_{r}\nabla^2_{y}\beta^{l}_{\delta}(\overline{v_m}_r)\nabla\overline{v_m}_r\psi\,dxdt\to\iint\rho\Delta\rho\phi_m(\rho)\nabla^2_{y}\beta^{l}_{\delta}(v_m)\nabla v_m\psi\,dxdt,\\
&\iint\overline{\rho\Delta\rho\phi_m(\rho)}_{r}\nabla_{y}\beta^{l}_{\delta}(\overline{v_m}_r)\nabla\psi\,dxdt\to
\iint\rho\Delta\rho\phi_m(\rho)\nabla_{y}\beta^{l}_{\delta}(v_m)\nabla\psi\,dxdt,
\end{aligned}
\end{equation*}
where in third limit the Dominated Convergence Theorem and a possible passage to subsequence is needed.
At this point the would like to pass in the limit as $m\to\infty$, but we cannot deal with the term 
\begin{equation*}
\iint\rho\Delta\rho\phi_m(\rho)\nabla^2_{y}\beta^{l}_{\delta}(v_m)\nabla v_m\psi\,dxdt.
\end{equation*}
Before, taking the limit in $m$ we need an additional truncation in $\rho$ at a different scale. Consider $\belr\psi$ as test function in \eqref{eq:ren}
 \begin{equation*}
 \begin{aligned}
&\iint\rho \bedm\belrp\partial_t\overline{\rho}_{r}\psi\,dxdt+\iint\rho \bedm\belr\partial_t\psi\,dxdt+\iint\rho u\bedm\belrp\nabla\regrho\psi\,dxdt\\
&+\iint\rho u\bedm\belr\nabla\psi\,dxdt-\iint\rrho\textrm{tr}(\MT)\phi'_m(\rho)\rho u\nabla_{y}\bedm\belr\psi\,dxdt\\
&-\!\!\iint\phi_{m}(\rho)\,M\,\nabla_{y}^2\bedm\nabla v_m\belr\psi\,dxdt-\!\!\iint\phi_{m}(\rho)M\nabla_{y}\bedm\belrp\nabla\regrho\psi\,dxdt\\
&-\iint\phi_{m}(\rho)\,M\,\nabla_{y}\bedm\belr\nabla\psi\,dxdt-\iint M\phi'_{m}(\rho)\nabla \rho\nabla_{y}\bedm\belr\psi\,dxdt\\
&-\iint N\phi_{m}(\rho)\nabla_{y}\bedm\belr\psi\,dxdt+\iint\nabla\rho\Delta\rho\phi_m(\rho)\nabla_{y}\bedm\belr\psi\,dxdt\\
&+\iint\rho\Delta\rho\phi'_m(\rho)\nabla\rho\nabla_{y}\bedm\belr\psi\,dxdt+\iint\rho\Delta\rho\phi_m(\rho)\nabla_{y}^2\bedm\nabla v_m\belr\psi\,dxdt\\
&+\iint\rho\Delta\rho\phi_m(\rho)\nabla_{y}\bedm\belrp\nabla\regrho\psi\,dxdt+\iint\rho\Delta\rho\phi_m(\rho)\nabla_{y}\bedm\belr\nabla\psi\,dxdt=0.
\end{aligned}
\end{equation*}
By using \eqref{eq:exint}, the definition of $\phi_{m}$, the bounds \eqref{eq:p5}-\eqref{eq:estvm} and the fact the for fixed $\delta$ and fixed $\lambda$ we have that $\beta^{l}$ and $\bar{\beta}_{\delta}$ are smooth and satisfy the bounds \eqref{eq:bed} and \eqref{eq:bel} we have that as $r\to 0$, after a possible passage to subsequence, by using the Dominated Convergence Theorem the following identity holds true,
 \begin{equation}\label{eq:pren2}
 \begin{aligned}
\!\!&\iint\rho \bedm\belrhop\partial_t\rho\psi\,dxdt+\iint\rho \bedm\belrho\partial_t\psi\,dxdt\\
\!\!&+\iint\rho u\bedm\belrhop\nabla\rho\psi\,dxdt+\iint\rho u\bedm\belrho\nabla\psi\,dxdt\\
\!\!&-\iint\rrho\textrm{tr}(\MT)\phi'_m(\rho)\rho u\nabla_{y}\bedm\belrho\psi\,dxdt-\iint\phi_{m}(\rho)\,M\,\nabla_{y}^2\bedm\nabla v_m\belrho\psi\,dxdt\\
\!\!&-\iint\phi_{m}(\rho)\,M\,\nabla_{y}\bedm\belrhop\nabla\rho\psi\,dxdt-\iint\phi_{m}(\rho)\,M\,\nabla_{y}\bedm\belrho\nabla\psi\,dxdt\\
\!\!&-\iint M\phi'_{m}(\rho)\nabla \rho\nabla_{y}\bedm\belrho\psi\,dxdt-\iint N\phi_{m}(\rho)\nabla_{y}\bedm\belrho\psi\,dxdt\\
\!\!&+\iint\nabla\rho\Delta\rho\phi_m(\rho)\nabla_{y}\bedm\belrho\psi\,dxdt+\iint\rho\Delta\rho\phi'_m(\rho)\nabla\rho\nabla_{y}\bedm\belrho\psi\,dxdt\\
\!\!&+\iint\rho\Delta\rho\phi_m(\rho)\nabla_{y}^2\bedm\nabla v_m\belrho\psi\,dxdt+\iint\rho\Delta\rho\phi_m(\rho)\nabla_{y}\bedm\belrhop\nabla\rho\psi\,dxdt\\
\!\!&+\iint\rho\Delta\rho\phi_m(\rho)\nabla_{y}\bedm\belrho\nabla\psi\,dxdt=0.
\end{aligned}
\end{equation}
 Now we send $m\to\infty$. Since $\log\rho\in L^{1}_{t,x}$, it follows that $\{\rho=0\}$ has measure zero and the following convergences holds
 \begin{equation}\label{eq:p20}
 \begin{aligned}
 &\phi_m(\rho)\to\,1\,\textrm{ a.e. in }(0,T)\times\T\textrm{ and }|\phi_m(\rho)|\leq 1,\\
 & v_m\to u \textrm{ a.e. in }(0,T)\times\T,\\
 &\rho\phi'_m(\rho)\to\,0\,\textrm{ a.e. in }(0,T)\times\T\textrm{ and }|\rho\phi'_m(\rho)|\leq 2,\\
 &\rrho\nabla v_m\to \MT\textrm{ strongly in }L^{2}_{t,x}.
 \end{aligned}
 \end{equation}
We only prove that last convergence, since the others are obtained directly from the definition of $\phi_m$. By the definition of $v_m$ and $\phi_m$ we have 
\begin{equation*}
\rrho\nabla\,v_m=\nabla\left(\frac{\phi_m(\rho)}{\rrho}\rho\,u\right)-\phi_m(\rho)\nabla\rho\otimes u.
\end{equation*}
By using that 
\begin{equation*}
\nabla(\rho\,u)=\rrho\MT+2\nabla\rrho\otimes\rrho u, 
\end{equation*}
we have that 
\begin{equation*}
\rrho\nabla\,v_m=\phi_m(\rho)\MT+4\phi'_m(\rho)\rho\nabla\rho^{\frac{1}{4}}\otimes\rho^{\frac{1}{4}}u. 
\end{equation*}
Then, 
\begin{equation*}
\begin{aligned}
\|\rrho\nabla\,v_m-\MT\|_{L^{2}_{t,x}}&\leq\|(\phi_{m}(\rho)-1)\MT\|_{L^{2}_{t,x}}\\
&+4\|\phi'_m(\rho)\rho\nabla\rho^{\frac{1}{4}}\otimes\rho^{\frac{1}{4}}u\|_{L^{2}_{t,x}},
\end{aligned}
\end{equation*}
and the right-hand side goes to zero by the Dominated Convergence Theorem and \eqref{eq:p20}. 
Next, we start to analyze the terms in \eqref{eq:pren2}. By using the definition of $M$ we have that  
 \begin{equation*}
 \begin{aligned}
 \iint\phi_{m}(\rho)\,M\,\nabla_{y}^2\bedm\nabla v_m\belrho\psi\,dxdt&=\iint \phi_{m}(\rho)\rrho\MT^{s}\nabla_{y}^2\bedm\nabla v_m\belrho\psi\,dxdt\\
 &+\e\iint \phi_{m}(\rho)\rrho\nabla^{2}\rrho \nabla_{y}^2\bedm\nabla v_m\belrho\psi\,dxdt\\
 &-4\e\iint\nabla\rho^{\frac{1}{4}}\otimes\nabla\rho^{\frac{1}{4}}\nabla_{y}^2\bedm\rrho\nabla v_m\belrho\psi\,dxdt,
 \end{aligned}
 \end{equation*}
 which thanks to \eqref{eq:exint}, the Dominated Convergence Theorem and \eqref{eq:p20} converge to 
 \begin{equation*}
 \begin{aligned}
&\iint \MT^{s}\MT\nabla_{y}^2\bed\belrho\psi\,dxdt+\e\iint\nabla^{2}\rrho \nabla_{y}^2\bed\MT\belrho\psi\,dxdt\\
&-4\e\iint\nabla\rho^{\frac{1}{4}}\otimes\nabla\rho^{\frac{1}{4}}\MT\nabla_{y}^2\bed\belrho\psi\,dxdt.
 \end{aligned}
 \end{equation*}
 Then, by using the definition of $M$ we have 
 \begin{equation*}
 \begin{aligned}
 \iint\phi_{m}(\rho)\,M\,\nabla_{y}\bedm\belrhop\nabla\rho\psi\,dxdt&= \iint\phi_{m}(\rho)\,\rrho\MT^{s}\nabla_{y}\bedm\belrhop\nabla\rho\psi\,dxdt\\
 &+\e\iint\phi_{m}(\rho)\rrho\nabla^{2}\rrho\nabla_{y}\bedm\belrhop\nabla\rho\psi\,dxdt\\
 &-\e\iint\phi_{m}(\rho)\nabla\rrho\otimes\nabla\rrho\nabla_{y}\bedm\belrhop\nabla\rho\psi\,dxdt,
 \end{aligned}
 \end{equation*}
 which by \eqref{eq:exint} and the Dominated Convergence Theorem converge to 
 \begin{equation*}
 \begin{aligned}
&\iint\,\rrho\MT^{s}\nabla_{y}\bed\belrhop\nabla\rho\psi\,dxdt+\e\iint\rrho\nabla^{2}\rrho\nabla_{y}\bed\belrhop\nabla\rho\psi\,dxdt\\
 &-\e\iint\nabla\rrho\otimes\nabla\rrho\nabla_{y}\bed\belrhop\nabla\rho\psi\,dxdt.
 \end{aligned}
 \end{equation*}
 Next, again by the definition of $M$ we have that 
 \begin{equation*}
 \begin{aligned}
\iint M\phi'_{m}(\rho)\nabla \rho\nabla_{y}\bedm\belrho\psi\,dxdt&=\iint \rrho\MT^{s}\phi'_{m}(\rho)\nabla \rho\nabla_{y}\bedm\belrho\psi\,dxdt\\
&+\e\iint\rrho\nabla^{2}\rrho\phi'_{m}(\rho)\nabla \rho\nabla_{y}\bedm\belrho\psi\,dxdt\\
&-\e\iint\nabla\rrho\otimes\nabla\rrho\phi'_{m}(\rho)\nabla \rho\nabla_{y}\bedm\belrho\psi\,dxdt\\
&=2\iint \MT^{s}\rho\phi'_{m}(\rho)\nabla \rrho\nabla_{y}\bedm\belrho\psi\,dxdt\\
&+2\e\iint\nabla^{2}\rrho\rho\phi'_{m}(\rho)\nabla \rrho\nabla_{y}\bedm\belrho\psi\,dxdt\\
&-4\e\iint\nabla\rho^{\frac{1}{4}}\otimes\nabla\rho^{\frac{1}{4}}\rho\phi'_{m}(\rho)\nabla \rrho\nabla_{y}\bedm\belrho\psi\,dxdt,
 \end{aligned}
 \end{equation*}
 which thanks to \eqref{eq:exint}, \eqref{eq:p20} and the Dominated Convergence Theorem converge to $0$. Finally, we consider the term 
 \begin{equation*}
 \iint\rho\Delta\rho\phi_m(\rho)\nabla_{y}^2\bedm\nabla v_m\belrho\psi\,dxdt.
 \end{equation*}
 In order to show the convergence of this term we need to use the additional truncation of $\rho$ at height $\lambda$. We have that 
 \begin{equation*}
 \begin{aligned}
 \iint\rho\Delta\rho\phi_m(\rho)\nabla_{y}^2\bedm\nabla v_m\belrho\psi\,dxdt&=\iint\rrho\Delta\rho\phi_m(\rho)\nabla_{y}^2\bedm\MT\belrho\psi\,dxdt\\
 &+\iint(\rrho\Delta\rho\nabla_{y}^2\bedm(\rrho\nabla v_m-\MT)\belrho\psi\,dxdt.
 \end{aligned}
 \end{equation*}
 Regarding the first term, by noticing that  
 \begin{equation*}
 |\rrho\Delta\rho\phi_m(\rho)\nabla_{y}^2\bedm\MT\belrho\psi|\leq \frac{C}{\sqrt{\lambda}}(|\Delta\rho|^2+|\MT|^2),
 \end{equation*}
 we have that the Dominated Convergence Theorem implies that 
 \begin{equation*}
 \iint\rrho\Delta\rho\phi_m(\rho)\nabla_{y}^2\bedm\MT\belrho\psi\,dxdt\to \iint\rrho\Delta\rho\nabla_{y}^2\bed\MT\belrho\psi\,dxdt.
 \end{equation*}
Concerning the second one we have that 
 \begin{equation*}
 \left|\iint(\rrho\Delta\rho\nabla_{y}^2\bedm(\rrho\nabla v_m-\MT)\belrho\psi\,dxdt\right|\leq C_{\lambda, \delta, \e}\|\Delta\rho\|_{L^{2}_{t,x}}\|\rrho\nabla v_m-\MT\|_{L^{2}_{t,x}}\to 0.
 \end{equation*}
For all the other terms in \eqref{eq:pren2} the analysis of the limit as $m\to\infty$ for fixed $\e$ is a consequence of the convergences of $\phi_m$ and $v_m$, a combination of the estimates in \eqref{eq:exint} and the Dominated Convergence Theorem. The equality \eqref{eq:ren211} is then proved for any $\psi\in C^{\infty}_{c}((0,T)\times\T;\R)$. The initial data can be recovered by using the weak continuity of $\rho\,u$ in Definition \ref{def:aws}, and by considering  $\chi_{n}(t)\psi(t,x)$ as a test function, with $\psi\in C^{\infty}([0,T);C^{\infty}_{c}(\T))$ and $\chi_n$ being an approximation of he Dirac Delta in $t=0$.\\
\bigskip

\noindent 2) Multiply \eqref{eq:diss} by $\overline{\phi_{m}(\rho)\bhmr}_{r}\varphi$ with $\varphi\in C^{\infty}_{c}((0,T)\times\T;\R)$ to obtain 
\begin{equation*}
\begin{aligned}
\iint \rrho\MT\overline{\phi_{m}(\rho)\bhmr}_{r}\varphi\,dxdt=&-\iint \phi_{m}(\rho)\bhmr\,\overline{\rho u\otimes\nabla\varphi}_{r}\,dxdt\\
&-\iint\overline{\rho u\varphi}_{r}\phi'_m(\rho)\nabla\rho\bhmr\,dxdt\\
&-\iint\overline{\rho u\varphi}_{r}\phi_m(\rho)\nablay\bhmr\nabla\overline{v_m}_{r}\,dxdt\\
&-2\iint\overline{\rrho u\otimes \nabla\rrho\varphi}_{r}\phi_m(\rho)\bhmr\,dxdt.
\end{aligned}
\end{equation*}
By sending $r\to 0$ and using \eqref{eq:exint} we easily get 
\begin{equation*}
\begin{aligned}
\iint \rrho\MT\phi_{m}(\rho)\bhm\varphi\,dxdt=&-\iint \phi_{m}(\rho)\bhm\rho u\otimes\nabla\varphi\,dxdt\\
&-\iint\rho u\varphi\phi'_m(\rho)\nabla\rho\bhm\,dxdt\\
&-\iint\rho u\varphi\phi_m(\rho)\nablay\bhm\nabla v_m\,dxdt\\
&-2\iint\rrho u\otimes \nabla\rrho\varphi\phi_m(\rho)\bhm\,dxdt.
\end{aligned}
\end{equation*}
Then, by sending $m\to \infty$ and using \eqref{eq:exint} and \eqref{eq:p20} we easily get \eqref{eq:dissren}.
 \end{proof}

\section{Global existence of weak solutions}\label{sec:main}
In this Section we are going to prove the main result of our paper. 
\subsection{ Bounds independent on $\e$}
We collect the $\e$-independent bounds from \eqref{eq:en} and \eqref{eq:bd1}, which we will use in the sequence. First, we have that, for a generic constant $C>0$ independent on $\e$, the following bounds hold true: 
\begin{equation}\label{eq:ub1} 
\begin{aligned}
&\|\sqrt{\rho_\e}u_\e\|_{L^\infty_tL^2_x}\leq C,\,&\|\nabla\re\|_{L^\infty_tL^2_x}\leq C,\quad&\|\rho_\e\|_{L^\infty_t(L^1_x\cap L^\gamma_x)}\leq C,\\
&\|\MT_{\e}\|_{L^2_{t, x}}\leq C,\,&\|\nabla\rho_n^{\gamma/2}\|_{L^2_{t, x}}\leq C,\,&\|\Delta\rn\|_{L^2_{t, x}}\leq C,\\
&\|\nabla\rre\|_{L^\infty_tL^2_x}\leq C.
\end{aligned}
\end{equation}
Moreover, 
\begin{equation}\label{eq:ub4} 
\begin{aligned}
\|\rho_\e\|_{L^2_tL^\infty_x}\leq C,\quad &
\|\nabla\rho_\e\|_{L^\frac{10}{3}_{t,x}}\leq C, \quad \|\rn^{\frac{\gamma}{2}}\|_{L^{\frac{10}{3}}_{t,x}}\leq C. 
\end{aligned}
\end{equation}
By using \eqref{eq:dissor}, \eqref{eq:ub1}, \eqref{eq:ub4} we have 
\begin{equation}\label{eq:ub5}
\begin{aligned}
&\|\rn\un\|_{L^{2}_{t,x}}\leq C,\qquad&\|\nabla(\rn\un)\|_{L^{2}_{t}L^{1}_{x}}\leq C. 
\end{aligned}
\end{equation}
By using the continuity equation \eqref{eq:wfce} and \eqref{eq:ub5} we have that 
\begin{equation}\label{eq:ub6}
\|\partial_t\rn\|_{L^2_tL^1_x}\leq C.
\end{equation}
Finally, from \eqref{eq:en} we also have that
\begin{equation}\label{eq:epsbound}
\begin{aligned}
&\quad\|\e^{\frac{1}{2}}\nabla^{2}\rre\|_{L^{2}_{t,x}}\leq C,\quad&\|\e^{\frac{1}{4}}\nabla\re^{\frac{1}{4}}\|_{L^{4}_{t,x}}\leq C,\\
&\quad \|\e^{\frac{1}{4}}\re^{\frac{1}{4}}\ue\|_{L^{4}_{t,x}},\quad&\|\sqrt{\e}\ue\|_{L^{2}_{t,x}}\leq C.
\end{aligned}
\end{equation}
\subsection{Convergence Lemma}
By using the above uniform bounds we prove the following convergences. 
\begin{lemma}\label{lem:c1}
Let $\{(\rn, \un,\MT_{\e})\}_\e$ be a sequence of weak solutions of \eqref{eq:app}-\eqref{eq:idapp}. 
\begin{enumerate}
\item Up to subsequences there exist, $\rho$, $m$, $\MT$ and $\Lambda$ such that  
\begin{align}
&\rn\rightarrow\rho\textrm{ strongly in }L^{2}(0,T;H^{1}(\T)),\label{eq:strong1}\\
&\rn\un\rightarrow m\textrm{ strongly in }L^{p}(0,T;L^{p}(\T))\textrm{ with }p\in[1,2),\label{eq:strong3}\\
&\MT_{\e}\weakto \MT\textrm{ weakly in }L^{2}((0,T)\times\T),\label{eq:weakvisc}\\
&\rrn\un\weaktos\Lambda\textrm{ weakly* in }L^{\infty}(0,T;L^{2}(\T)).\label{eq:weakrru}
\end{align}
Moreover, $\Lambda$ is such that $\rrho\Lambda=m$. 
\item The following additional convergences hold true for the density
 \begin{align}
&\nabla\rrn\weakto\nabla\rrho\textrm{ weakly in }L^{2}((0,T)\times\T),\label{eq:weaknrr}\\
&\Delta\rn\weakto\Delta\rho\textrm{ weakly in }L^{2}((0,T)\times\T),\label{eq:weakdr}\\
&\rn^{\gamma}\rightarrow\rho^{\gamma}\textrm{ strongly in }L^{1}((0,T)\times\T),\label{eq:strong4}\\
&\nabla\rn^{\frac{\gamma}{2}}\weakto\nabla\rho^{\frac{\gamma}{2}}\textrm{ weakly in }L^{2}((0,T)\times\T).\label{eq:weaknp}
\end{align}
\end{enumerate}
\end{lemma}
\begin{proof}
By using \eqref{eq:qns1} and \eqref{eq:ub4}, we have that
\begin{equation*}
\{\partial_t\rn\}_{\e}\textrm{ is uniformly bounded in }L^{2}(0,T;H^{-1}(\T)).
\end{equation*} 
Then, since $\{\rn\}_{\e}$ is uniformly bounded in $L^{2}(0,T;H^{2}(\T))$, by using Aubin-Lions Lemma we get \eqref{eq:strong1}. Next, by using the momentum equations and the bounds \eqref{eq:ub1}-\eqref{eq:ub4}, it is easy to prove that 
\begin{equation*}
\{\partial_t(\rn\un)\}_\e\textrm{ is uniformly bounded in }L^{2}(0,T;W^{-2,\frac{3}{2}}(\T)).
\end{equation*}
Then, by using \eqref{eq:ub4}, \eqref{eq:ub5} and Aubin-Lions Lemma, \eqref{eq:strong3} follows. The convergences \eqref{eq:weakvisc} and \eqref{eq:weakrru} follow by standard weak compactness theorems and the equality $\rrho\Lambda=m$ follows easily from \eqref{eq:strong1} and \eqref{eq:weakrru}. Next, the convergences \eqref{eq:weaknrr}, \eqref{eq:weakdr}  follow from the the uniform bounds \eqref{eq:ub1} and standard weak compactness arguments. Finally, The convergence \eqref{eq:strong4} is obtained by using \eqref{eq:strong1} and the bound \eqref{eq:ub1} and the convergence \eqref{eq:weaknp} follows by \eqref{eq:ub1} and \eqref{eq:strong1}.
\end{proof}
\begin{lemma}\label{lem:c2}
Let $f\in C\cap L^{\infty}(\R^{3};\R)$ and $(\rn, \un)$ be a solution of \eqref{eq:qns1}-\eqref{eq:qns2} and let $u$ be defined as follows: 
\begin{equation}\label{eq:defu}
u=
\left\{
\begin{array}{rll}
&\frac{m(t,x)}{\rho(t,x)}=\frac{\Lambda(t,x)}{\sqrt{\rho(t,x)}} & (t,x)\in \{\rho>0\}, \\
&0 & (t,x)\in\{\rho=0\}.
\end{array}
\right.
\end{equation}
Then, the following convergences hold. 
\begin{align}
\rn\,f(u_{\e})\to\rho\,f(u)&\textrm{ strongly in }L^{p}((0,T)\times\T)\textrm{ for any }p<6,\label{eq:convro}\\
\nabla\rn\,f(u_{\e})\to \nabla\rho\,f(u)&\textrm{ strongly in }L^{p}((0,T)\times\T)\textrm{ for any }p<\frac{10}{3},
\label{eq:convnro}\\
\rn\un\,f(\un)\to \rho u\,f(u)&\textrm{ strongly in }L^{p}((0,T)\times\T)\textrm{ for any }p<2,\label{eq:convm}\\
\rn^{\frac{\gamma}{2}}\,f(\un)\to \rho^{\frac{\gamma}{2}}\,f(u)&\textrm{ strongly in }L^{p}((0,T)\times\T)\textrm{ for any }p<\frac{10}{3}.\label{eq:convp}
\end{align}
\end{lemma}
\begin{proof}
We first note that, up to a subsequence non relabelled, \eqref{eq:strong1} and \eqref{eq:strong3} imply that 
\begin{equation}\label{eq:convrom}
\begin{aligned}
&\rn\to\rho\textrm{ a.e. in }(0,T)\times\T,\\
&\rn\un\to m\textrm{ a.e. in }(0,T)\times\T,\\
&\nabla\rn\to\nabla\rho\textrm{ a.e. in }(0,T)\times\T.
\end{aligned}
\end{equation}
Moreover, by Fatou Lemma we have that 
\begin{equation}
\iint\liminf_{\e\to0}\frac{m_{\e}^2}{\rn}\,dxdt\leq \liminf_{\e\to0}\iint\frac{m_{\e}^2}{\rn}<\infty,
\end{equation}
which implies that $m=0$ on $\{\rho=0\}$ and 
\begin{equation*}
\rrho\,u\in L^{\infty}(0,T;L^{2}(\T)).
\end{equation*}
Moreover, $m=\rho\,u=\rrho\Lambda$. 
Let us prove \eqref{eq:convro}. On $\{\rho>0\}$ by using \eqref{eq:convrom} we have that 
\begin{equation*}
\rn\,f(u_{\e})\to\rho\,f(u)\textrm{ a.e. in }\{\rho>0\}.
\end{equation*}
On the other hand, since $f\in L^{\infty}(\R^{3};\R)$ we have 
$$
|\rn\,f(u_{\e})|\leq |\rn|\|f\|_{\infty}\to 0\textrm{ a.e. in }\{\rho=0\}. $$
Then, $\rn\,f(u_{\e})\to\rho\,f(u)$ a.e. in $(0,T)\times\T$ and the convergence in \eqref{eq:convro} follows by the uniform bound 
\begin{equation*}
\|\rn\|_{L^{6}_{t,x}}\leq C
\end{equation*}
and Vitali's Theorem. 
Regarding \eqref{eq:convnro}, from Lemma \ref{lem:c1} we have that $\rho$ is a Sobolev function, then, see \cite{EG}, 
\begin{equation*}
\nabla\rho=0\textrm{ a.e. in }\{\rho=0\}. 
\end{equation*}
From \eqref{eq:convrom} we have that 
\begin{equation*}
\begin{aligned}
&\nabla\rn\,f(u_{\e})\to\nabla\rho\,f(u)\textrm{ a.e. in }\{\rho>0\}\\
&|\nabla\rn\,f(u_{\e})|\leq |\nabla\rn|\|f\|_{\infty}\to 0\textrm{ a.e. in }\{\rho=0\}.
\end{aligned}
\end{equation*}
Then, $\nabla\rn\,f(u_{\e})\to\nabla\rho\,f(u)$ a.e. in $(0,T)\times\T$ and \eqref{eq:convnro} follows from the uniform bound \eqref{eq:ub4} and Vitali's Theorem. Concerning \eqref{eq:convm}, again \eqref{eq:convrom} implies the following convergences
\begin{equation*}
\begin{aligned}
&\rn\un\,f(u_{\e})\to m\,f(u)\textrm{ a.e. in }\{\rho>0\},\\
&|\rn\un\,f(u_{\e})|\leq |\rn\un|\|f\|_{\infty}\to 0\textrm{ a.e. in }\{\rho=0\},
\end{aligned}
\end{equation*}
which, together with \eqref{eq:ub4} and Vitali's Theorem, imply \eqref{eq:convm}. Finally, \eqref{eq:convp} follows by the same arguments used to prove \eqref{eq:convro} and the uniform bounds on the pressure in \eqref{eq:ub1}.  
\end{proof}
\subsection{Proof of the main Theorem}
We are now ready to prove Theorem \ref{teo:main}. 
\begin{proof}[Proof of Theorem \ref{teo:main}]
Let $\{(\rn,\un,\MT_{\e})\}_{\e}$ be a sequence of weak solutions of \eqref{eq:app}-\eqref{eq:idapp}. By Lemma \ref{lem:c1} there exist $\rho$, $m$, $\Lambda$  and $\MT$ such that the convergences \eqref{eq:strong1}, \eqref{eq:strong3} and \eqref{eq:weakrru} hold. Moreover, 
by defining the velocity $u$ as in Lemma \ref{lem:c2} we have that 
\begin{equation*}
\begin{aligned}
&\rrho\,u\in L^{\infty}(0,T;L^{2}(\T),\quad &\MT\in L^{2}((0,T)\times\T),\quad
&m=\rrho\Lambda=\rho\,u. 
\end{aligned}
\end{equation*}
By using \eqref{eq:strong1}, \eqref{eq:strong3} and \eqref{eq:hyidr} is straightforward to prove that 
\begin{equation*}
\int\rn^0\phi(0,x)\,dx+\iint\rn\phi_t\,dxdt+\iint\rn\un\nabla\phi\,dxdt
\end{equation*}
converges to 
\begin{equation*}
\int\rho^0\phi(0,x)\,dx+\iint\rho\phi_t\,dxdt+\iint\rho\,u\nabla\phi\,dxdt,
\end{equation*}
for any $\phi\in C^{\infty}_{c}([0,T)\times\T)$. 
Let us consider the momentum equations.  Let $l\in\{1,2,3\}$ fixed, by using Theorem \ref{teo:exmain} we have that 
for any $\psi\in C^{\infty}_{c}([0,T)\times\T;\R)$ the following equality holds
\begin{equation}\label{eq:ren2}
\begin{aligned}
&\int\rho^{0}\beta_{\delta}^{l}(u^0)\bar{\beta}_{\la}(\rho^0)\psi(0,x)\,dx+\iint\rn\bedn\beln\partial_t\psi-\iint\rn \ue\bedn\beln\cdot\nabla\psi\,dxdt\\
&-\iint\rre\MT^{s}_{\e}:\nablay\bedn\beln\otimes\nabla\psi\,dxdt-2\iint\re^{\frac{\gamma}{2}}\nabla\re^{\frac{\gamma}{2}}\cdot\nabla_y\bedn\beln\psi\,dxdt\\
&-\iint\nabla\re\Delta\re\nablay\bedn\beln\psi\,dxdt-\iint\re\Delta\re\nablay\bedn\beln\nabla\psi\,dxdt\\
&+\iint R^{\delta,\la}_{\e}\psi\,dxdt+\iint \tilde{R}^{\delta,\la}_{\e}\psi\,dxdt=0.
\end{aligned}
\end{equation}
where the remainders are
\begin{equation*}
\begin{aligned}
R^{\delta,\la}_{\e}=\sum_{i=1}^{6}R^{\delta,\la}_{\e,i}&=\rn\bedn\belnp\partial_t\rn+\rn\un\bedn\belnp\nabla\rn\\
&-\rre\MT^{s}_e:\nabla_{y}\bedn\otimes\nabla\re\belnp+\rre\Delta\re\nabla_{y}^2\bedn\MT_{\e}\beln\\
\\
&+\rn\Delta\re\nablay\bedn\belnp\nabla\rn-\MT_{\e}^{s}\MT_{\e}\nabla_{y}^2\bedn\beln,\\
\\
\tilde{R}^{\delta,\la}_{\e}=\sum_{i=1}^{4}\tilde{R}^{\delta,\la}_{\e,i}&=-\e\nabla^{2}\rre\,\MT_{\e}\,\nabla^{2}_{y}\bedn\beln+4\e\nabla\rn^{\frac{1}{4}}\otimes\nabla\rn^{\frac{1}{4}}\,\MT\nabla^{2}_{y}\bedn\beln\\
&-\e\rre\nabla^{2}\rre\nabla_{y}\bedn\belnp\nabla\re+4\e\rre\nabla\re^{\frac{1}{4}}\otimes\nabla\re^{\frac{1}{4}}\bedn\belnp\nabla\rn\\
\\
&-\e\re|\ue|^{2}\ue\nabla_{y}\bedn\beln-\e \ue\nabla_{y}\bedn\beln.
\end{aligned}
\end{equation*}
We first perform the limit as $\e$ goes to $0$ for $\delta$ and $\la$ fixed. Notice that, since $\bar{\beta}_{\lambda}\in L^{\infty}(\R)$, and $\{\rn\}_{\e}$ converges almost everywhere, by Dominated Convergence we have that
\begin{equation}\label{eq:convrot}
\beln\to\bar{\beta}_{\lambda}(\rho)\textrm{ strongly in }L^{q}((0,T)\times\T)\textrm{ for any }q<\infty.
\end{equation}
By using \eqref{eq:convro} with $p=2$ and choosing $q=2$ in \eqref{eq:convrot} we have that 
\begin{equation*}
\iint\rn\bedn\beln\partial_t\psi\,dxdt\to\iint\rho\bed\bel\partial_t\psi\,dxdt.
\end{equation*}
Next, by \eqref{eq:convm} with $p=3/2$ and choosing $q=3$ in \eqref{eq:convrot} we get
\begin{equation*}
\iint\rn\un\bedn\beln\cdot\nabla\psi\,dxdt\to \iint\rho\,u\bed\bel\cdot\nabla\psi\,dxdt.
\end{equation*}
By using \eqref{eq:weakvisc}, \eqref{eq:convro} with $p=4$ and \eqref{eq:convrot} with $q=4$ it follows
\begin{equation*}
\iint\MT_{\e}^{s}:\rrn\nablay\bedn\beln\otimes\nabla\psi\,dxdt\to\iint\rrho\,\MT:\nablay\bed\bel\otimes\nabla\psi\,dxdt.
\end{equation*}
By using \eqref{eq:weaknp}, \eqref{eq:convp} with $p=3$ and \eqref{eq:convrot} with $q=6$ it follows
\begin{equation*}
\iint\rn^{\frac{\gamma}{2}}\nabla\rn^{\frac{\gamma}{2}}\cdot\nablay\bedn\beln\psi\,dxdt\to
\iint\rho^{\frac{\gamma}{2}}\nabla\rho^{\frac{\gamma}{2}}\cdot\nablay\bed\bel\psi\,dxdt.
\end{equation*}
By using \eqref{eq:weakdr}, \eqref{eq:convnro} with $p=3$ and \eqref{eq:convrot} with $q=6$ it follows
\begin{equation*}
\iint\nabla\rn\Delta\rn\nablay\bedn\beln\psi\,dxdt\to\iint\nabla\rho\Delta\rho\nablay\bed\bel\psi\,dxdt.
\end{equation*}
Next, by using \eqref{eq:weakdr}, \eqref{eq:convro} with $p=3$ and \eqref{eq:convrot} with $q=6$ it follows
\begin{equation*}
\iint\rn\Delta\rn\nablay\bedn\beln\nabla\psi\,dxdt\to\iint\rho\Delta\rho\nablay\bed\bel\nabla\psi\,dxdt.
\end{equation*}
It remains to study the remainders $\tilde{R}^{\delta,\la}_{\e}$ and ${R}^{\delta,\la}_{\e}$. 
Regarding  $\tilde{R}^{\delta,\la}_{\e}$ we prove the following convergence 
\begin{equation*} 
\tilde{R}^{\delta,\la}_{\e}\to 0\mbox{ in }L^{1}_{t,x}.
\end{equation*}
Indeed, by considering term by term and using the uniform bounds \eqref{eq:ub1} and \eqref{eq:epsbound} we have that 
\begin{equation*}
\begin{aligned}
&\|\tilde{R}^{\delta,\la}_{\e,1}\|_{L^{1}_{t,x}}\leq C_{\delta,\la}\sqrt{\e}\|\sqrt{\e}\nabla^{2}\rre\|_{L^{2}_{t,x}}\|\MT_{\e}\|_{L^{2}_{t,x}}\leq C_{\delta,\la}\sqrt{\e},\\
&\|\tilde{R}^{\delta,\la}_{\e,2}\|_{L^{1}_{t,x}}
\leq C_{\delta,\la}\sqrt{\e}\|\e^{\frac{1}{4}}\nabla\re^{\frac{1}{4}}\|^{2}_{L^{4}_{t,x}}\|\MT_{\e}\|_{L^{2}_{t,x}}\leq C_{\delta,\la}\sqrt{\e},\\
&\|\tilde{R}^{\delta,\la}_{\e,3}\|_{L^{1}_{t,x}}\leq C_{\delta,\la}\sqrt{\e}\|\sqrt{\e}\nabla^{2}\rre\|_{L^{2}_{t,x}}\|\nabla\re\|_{L^{2}_{t,x}}\leq C_{\delta,\la}\sqrt{\e},\\
&\|\tilde{R}^{\delta,\la}_{\e,4}\|_{L^{1}_{t,x}}\leq C_{\delta,\la}\sqrt{\e}\|\e^{\frac{1}{4}}\nabla\re^{\frac{1}{4}}\|^{2}_{L^{4}_{t,x}}\|\nabla\re\|_{L^{2}_{t,x}}\leq C_{\delta,\la}\sqrt{\e},\\
&\|\tilde{R}^{\delta,\la}_{\e,5}\|_{L^{1}_{t,x}}\leq C_{\delta,\la}\e^{\frac{1}{4}}\|\rho\|_{L^{1}_{t,x}}^{\frac{1}{4}}\|\e^{\frac{1}{4}}\re^{\frac{1}{4}}\ue\|_{L^{4}_{t,x}}^{3}\leq C_{\delta,\la}\e^{\frac{1}{4}},\\
&\|\tilde{R}^{\delta,\la}_{\e,6}\|_{L^{1}_{t,x}}\leq C_{\delta,\la}\sqrt{\e}\|\sqrt{\e}\ue\|_{L^{2}_{t,x}}\leq C_{\delta,\la}\sqrt{\e}.
\end{aligned}
\end{equation*}
Now we consider $R^{\delta,\la}_{\e}$.
We claim that there exists a $C>0$ independent on $\e$, $\delta$ and $\la$ such that 
\begin{align}
&\|R^{\delta,\la}_{\e}\|_{L^{1}_{t,x}}\leq C\left(\frac{\delta}{\sqrt{\la}}+\frac{\la}{\delta}+\la+\delta\right).\label{eq:r1}
\end{align}
In order to prove \eqref{eq:r1} we estimate all the terms in \eqref{eq:remainder1} separately. By using the uniform bounds \eqref{eq:ub1}, \eqref{eq:ub4}, \eqref{eq:ub6} and the bounds on the truncations \eqref{eq:bed} and \eqref{eq:bel} we have
\begin{equation*}
\begin{aligned}
&\|R^{\delta,\la}_{\e,1}\|_{L^{1}_{t,x}}\leq \|\rn\|_{L^{2}(L^{\infty})}\|\partial_{t}\rn\|_{L^{2}(L^{1})}\|\bedn\|_{L^{\infty}_{t,x}}
\|\belnp\|_{L^{\infty}_{t,x}}\leq C\frac{\lambda}{\delta},\\
&\|R^{\delta,\la}_{\e,2}\|_{L^{1}_{t,x}}\leq \|\rn\un\|_{L^{2}_{t,x}}\|\nabla\rn\|_{L^{2}_{t,x}}\|\bedn\|_{L^{\infty}_{t,x}}
\|\belnp\|_{L^{\infty}_{t,x}}\leq C\frac{\lambda}{\delta},\\
&\|R^{\delta,\la}_{\e,3}\|_{L^{1}_{t,x}}\leq \|\rn\|_{L^{2}(L^{\infty})}\|\MT^{s}_{\e}\|_{L^{2}_{t,x}}\|\nabla\rn\|_{L^{\infty}(L^{2})}
\|\nablay\bedn\|_{L^{\infty}_{t,x}}\|\belnp\|_{L^{\infty}_{t,x}}\leq C\la,\\
&\|R^{\delta,\la}_{\e,4}\|_{L^{1}_{t,x}}\leq \|\Delta\rn\|_{L^{2}_{t,x}}\|\MT^{s}_{\e}\|_{L^{2}_{t,x}}\|\nablay^2\bedn\|_{L^{\infty}_{t,x}}\|\rrn\beln\|_{L^{\infty}_{t,x}}\leq C\frac{\delta}{\sqrt{\la}},\\
&\|R^{\delta,\la}_{\e,5}\|_{L^{1}_{t,x}}\leq \|\rn\|_{L^{2}(L^{\infty})}\|\Delta\rn\|_{L^{2}_{t,x}}\|\nabla\rn\|_{L^{\infty}(L^{2})}
\|\nablay\bedn\|_{L^{\infty}_{t,x}}\|\belnp\|_{L^{\infty}_{t,x}}\leq C\la,\\
&\|R^{\delta,\la}_{\e,6}\|_{L^{1}_{t,x}}\leq \|\MT_{\e}\|_{L^{2}_{t,x}}^2\|\nablay^2\bedn\|_{L^{\infty}_{t,x}}\|\beln\|_{L^{\infty}_{t.x}}\leq C\delta.
\end{aligned}
\end{equation*}
Then, \eqref{eq:r1} is proved and, when $\e$ goes to $0$, we have that $(\rho, u, \MT)$ satisfies the following integral equality
\begin{equation}\label{eq:ren3}
\begin{aligned}
&\iint\rho\,u\bed\bel\cdot\nabla\psi\,dxdt-2\nu\iint\sqrt{\rho}\mathcal{S}:\nablay\bed\bel\otimes\nabla\psi\,dxdt\\
&-\iint\rho^{\frac{\gamma}{2}}\nabla\rho^{\frac{\gamma}{2}}\cdot\nablay\bed\bel\psi\,dxdt\\
&-2\kappa^2\iint\nabla\rho\Delta\rho\nablay\bed\bel\psi\,dxdt-2\kappa^2\iint\rho\Delta\rho\nablay\bed\bel\nabla\psi\,dxdt\\
&-\int\rho^{0}\beta_{\delta}^{l}(u^0)\bar{\beta}_{\la}(\rho^0)\psi(0,x)\,dx+\langle\mu^{\delta,\la},\psi\rangle=0,
\end{aligned}
\end{equation}
where  $\mu^{\delta,\la}$ is a measure such that 
\begin{equation*}
\begin{aligned}
&R^{\delta,\la}_{\e}\to\mu^{\delta,\la}\textrm{ in }\mathcal{M}(\T; \R)
\end{aligned}
\end{equation*}
and its total variations satisfies
\begin{equation}\label{eq:totalvar}
\begin{aligned}
&|\mu^{\delta,\la}|(\T)\leq C\left(\frac{\delta}{\sqrt{\la}}+\frac{\la}{\delta}+\la+\delta\right).
\end{aligned}
\end{equation}
Let $\delta=\la^{\alpha}$ with $\alpha\in (1/2,1)$, then when $\la\to0$ we have that 
\begin{equation*}
\begin{aligned}
&|\mu^{\la^{\alpha},\la}|(\T)\to 0
\end{aligned}
\end{equation*}
and by \eqref{eq:bedc}, \eqref{eq:belc} and the Dominated Convergence Theorem we have that \eqref{eq:ren3} converges to
 \begin{equation}\label{eq:ren4}
\begin{aligned}
&\int\rho^{0}\,u^{l,0}\psi(0,x)\,dx+\iint\rho\,u^{l}\partial_t\psi+\iint\rho\,u\,u^{l}\cdot\nabla\psi\,dxdt
-\iint\sqrt{\rho}\MT^{s}_{lj}\nabla_{j}\psi\,dxdt\\&-\iint\rho^{\frac{\gamma}{2}}\nabla_{l}\rho^{\frac{\gamma}{2}}\psi\,dxdt
-\iint\nabla_{l}\rho\Delta\rho\psi\,dxdt-\iint\rho\Delta\rho\nabla_{l}\psi\,dxdt=0.
\end{aligned}
\end{equation}
It remain to prove \eqref{eq:dissor}. By using Theorem \ref{teo:exmain}, Part $2)$ we have that for any $\varphi\in C^{\infty}_{c}((0,T)\times\T;\R)$ it holds that 
\begin{equation}\label{eq:dissrenp}
\begin{aligned}
\iint \rre\MT_{\e}\bhue\varphi\,dxdt=&-\iint \bhue\re\ue\otimes\nabla\varphi\,dxdt\\
&-\iint\rre \ue\varphi\nablay\bhue\MT_{\e}\,dxdt\\
-&2\iint\rre \ue\otimes \nabla\rre\varphi\bhue\,dxdt.
\end{aligned}
\end{equation}
For fixed $\delta$, by using the convergence \eqref{eq:weakvisc} and \eqref{eq:convro} with $p=4$, we have that 
\begin{equation*}
\iint\rre \bhue\MT_{\e}\varphi\,dxdt\to \iint\rrho\hat{\beta}_{\delta}(u)\MT\varphi\,dxdt.
\end{equation*}
Next, we have that 
\begin{equation*}
\begin{aligned}
\iint \bhue\re\ue\otimes\nabla\varphi \,dxdt&\to\iint \bhu\rho u\otimes\nabla\varphi \,dxdt,
\end{aligned}
\end{equation*}
because of  \eqref{eq:convm} with $p=1$. By using \eqref{eq:ybed}, \eqref{eq:convro} with $p=2$ and the weak convergence of $\nabla\rrn$ in $L^{2}_{t,x}$ we get 
\begin{equation*}
\begin{aligned}
\iint\rre \ue\otimes \nabla\rre\bhue \varphi\,dxdt&\to
\iint\rrho u\otimes \nabla\rrho\bhu \varphi\,dxdt.
\end{aligned}
\end{equation*}
Let 
\begin{equation}\label{eq:remainder2}
\bar{R}^{\delta}_{\e}=\rre \ue\varphi\nablay\bhue\MT_{\e},
\end{equation}
by using \eqref{eq:ub1} and \eqref{eq:ybed} we have that 
\begin{equation*}
\|\bar{R}^{\delta}_{n}\|_{L^{1}_{t,x}}\leq C\|\rrn\ue\|_{L^{\infty}(L^{2}_{t,x})}\|\MT_{\e}\|_{L^{2}_{t,x}}\|\nablay\hat{\beta}_{\delta}(\un)\|_{L^{\infty}_{t,x}}\leq C\delta,
\end{equation*}
and then there exists a measure $\bar{\mu}^{\delta}$ such that 
\begin{equation}\label{eq:totvar2}
\iint\bar{R}^{\delta}_{\e}\,\nabla\varphi\,dxdt
\to\langle\bar{\mu}^{\delta},\nabla\varphi\rangle,
\end{equation}
and its total variation satisfies
\begin{equation*}
|\bar{\mu}^{\delta}|(\T)\leq C\delta.
\end{equation*}
Collecting the previous convergences, we have
\begin{equation*}
\begin{aligned}
\iint\rre \bhue\MT_{\e}\varphi\,dxdt\,&=-\iint \bhu\rho u\otimes\nabla\varphi \,dxdt\\
&-2\iint\rrho u\otimes \nabla\rrho\bhu \varphi\,dxdt\\
&-\langle\bar{\mu}^{\delta},\nabla\psi\rangle.
\end{aligned}
\end{equation*}
By using \eqref{eq:bedc}, Dominated Convergence Theorem and \eqref{eq:totvar2} we get \eqref{eq:dissor}.
Finally, the energy inequality follows from the lower semicontinuity of the norms. 
\end{proof}


\begin{thebibliography}{10}

\bibitem{AHM} P. Antonelli, L. Hientzsch and P. Marcati. {\em On the Low Mach number limit for Quantum-Navier-Stokes equations}, Preprint: arXiv:1902.00402. 

\bibitem{AM} P. Antonelli and P. Marcati, {\em On the finite energy weak solutions to a system in quantum fluid dynamics,} Comm. Math. Phys., {\bf287} (2009), 657--686.

\bibitem{AM2} P. Antonelli and P. Marcati, \emph{The Quantum Hydrodynamics system in two space dimensions}, Arch. Ration. Mech. Anal., \textbf{203} (2012), 499--527.

\bibitem{AS} P. Antonelli and S. Spirito, {\em On the compactness of finite energy weak solutions to the quantum Navier-Stokes equations}, J. Hyperbolic Differ. Equ., \textbf{15} (2018), 133--147.

\bibitem{AS1} P. Antonelli and S. Spirito, \emph{Global existence of finite energy weak solutions of quantum Navier-Stokes equations}, Arch. Ration. Mech. Anal., {\bf3} (2017), 1161--1199.

\bibitem{AS3} P. Antonelli and s. Spirito, \emph{On the compactness of finite energy weak solutions to the Navier-Stokes-Korteweg system with capillary effects}, Preprint: arXiv:1808.03495. 

\bibitem{AH} C. Audiard and B. Haspot, {\em Global well-posedness of the Euler-Korteweg system for small irrotational data}, Comm. Math. Phys., \textbf{351} (2017), 201--247.

\bibitem{BGDD} S. Benzoni-Gavage, R. Danchin and S. Descombes, \emph{On the well-posedness for the Euler-Korteweg model in several space dimensions}, Indiana Univ. Math. J., {\bf 56} (2007), 1499--1579.

\bibitem{BCNV} D. Bresch, F. Couderc, P. Noble, J.P. Vila, \emph{A generalization of the quantum Bohm identity: Hyperbolic CFL condition for the Euler-Korteweg equations, G\'en\'eralisation de l'identit\'e de Bohm quantique : condition CFL hyperbolique pour \'equations d'Euler–Korteweg.}, Comptes Rendus Math. {\bf 354}, no. 1 (2016), 39--43.

\bibitem{BD} D. Bresch and D. Desjardins, {\em Sur un mod\`ele de Saint-Venant visqueux et sa limite quasi-g\'eostrophique. [On viscous shallow-water equations (Saint-Venant model) and the quasi-geostrophic limit.]}, C. R. Math. Acad. Sci. Paris, {\bf335} 2002, 1079--1084.

\bibitem{BDL} D. Bresch, B. Desjardins and Chi-Kun Lin, {\em On some compressible fluid models: Korteweg, lubrication, and shallow water systems}, Comm. Part. Differ. Equat., {\bf28} (2003), 843--868.

\bibitem{BJ} D. Bresch and P.-E. Jabin, \emph{Global existence of weak solutions for compressible Navier-Stokes equations; thermodinamically unstable pressure and anisotropic viscous stress tensor}, Ann. of Math. (2), {\bf188} (2018), no. 2, 577--684

\bibitem{BGL} D. Bresch, M. Gisclon and I. Lacroix-Violet, {\em On Navier-Stokes-Korteweg and Euler-Korteweg Systems: Application to Quantum Fluids Models}. Preprint: arXiv:1703.09460.

\bibitem{CCH} R. Carles, K. Carrapatoso and M. Hillairet, {\em Rigidity results in generalized isothermal fluids}. Preprint: arXiv:1803.07837.

\bibitem{DPL} {\sc R.~J.~DiPerna and P.~L.~Lions,}
{\em Ordinary differential equations, transport theory and Sobolev spaces.}
Invent. Math., {\bf 98} (1989), 511--547.

\bibitem{DFM} D. Donatelli, E. Feireisl and P. Marcati,{\em Well/ill posedness for the Euler-Korteweg-Poisson system and related problems}, Comm. Part. Differ. Equat., \textbf{40} (2015), 1314--1335.

\bibitem{DM} D. Donatelli and P. Marcati, {\em Quasineutral limit, dispersion and oscillations for Korteweg type fluids}, SIAM J. Math. Anal., \textbf{47} (2015), 2265--2282.

\bibitem{DM1} D. Donatelli and P. Marcati, {\em Low {M}ach number limit for the quantum hydrodynamics system},
Res. Math. Sci., \textbf{3} (2016), 2522--0144.

\bibitem{DS} J.E. Dunn and J. Serrin, \emph{On the thermomechanics of interstitial working}, Arch. Ration. Mech. Anal., {\bf 88} (1985), no. 2, 95--133.

\bibitem{EG} L.C. Evans and R.F. Gariepy, {\em Measure theory and fine properties of functions}, Studies in Advanced Mathematics, CRC Press, Boca Raton, FL, 1992.

\bibitem{F} E. Feireisl, {\em On compactness of solutions to the compressible isentropic Navier-Stokes equations when the density is not square integrable}, Comment. Math. Univ. Carolin., {\bf42} (2001), 83--98.

\bibitem{GLV} M. Gisclon and I. Lacroix-Violet, {\em About the barotropic compressible quantum Navier-Stokes}
Nonlinear Anal., {\bf128} (2015),106--121.

\bibitem{GLT} J. Giesselmann, C. Lattanzio and A.-E. Tzavaras. {\em Relative energy for the Korteweg theory and related Hamiltonian flows in gas dynamics}, Arch. Ration. Mech. Anal., \textbf{223} (2017), 1427--1484.

\bibitem{J} A. J\"ungel, {\em Global weak solutions to compressible Navier-Stokes equations for quantum fluids}, SIAM J. Math. Anal., {\bf42} (2010), 1025--1045.

\bibitem{HL}  H. Hattori and D. Li, {\em Solutions for two dimensional system for materials of Korteweg type}, SIAM J. Math. Anal., {\bf25} (1994), 85--98.

\bibitem{HL1}H. Hattori and D. Li, {\em Global solutions of a high dimensional system for Korteweg materials}, J. Math. Anal. Appl., {\bf198} (1996), 84--97.

\bibitem{LV} I. Lacroix-Violet and A. Vasseur, {\em Global weak solutions to the compressible quantum Navier-Stokes equation and its semi-classical limit}, J. Math. Pures Appl., {\bf114} (2017), 191--210.

\bibitem{LX} J. Li and Z. Xin, {\em Global Existence of Weak Solutions to the Barotropic Compressible Navier-Stokes Flows with Degenerate Viscosities}, Preprint: arXiv:1504.06826. 

\bibitem{L} P.L. Lions, {\em Mathematical Topics in Fluid Mechanics. Vol. 2.}, Claredon Press, Oxford Science Publications, 1996.

\bibitem{MV} A. Mellet and A. Vasseur, {\em On the barotropic compressible Navier-Stokes equations}, Comm. Part. Differ. Equat., {\bf32} (2007), 431--452.

\bibitem{VY1} A. Vasseur and C. Yu, {\em Existence of global weak solutions for 3D degenerate compressible Navier-Stokes equations}, Invent. Math., {\bf206} (2015), 935--974.

\bibitem{VY2}A. Vasseur and C. Yu, {\em Global weak solutions to compressible quantum Navier-Stokes equations with damping}, SIAM J. Math. Anal., {\bf48} (2016), 1489--1511.
\end{thebibliography}
\end{document}